\documentclass{amsart}
\usepackage{wasysym}
\usepackage{amsmath}
\usepackage{amssymb}
\usepackage{amsthm}
\usepackage{tikz}
\usetikzlibrary{matrix,arrows,backgrounds,shapes.misc,shapes.geometric,patterns,calc,positioning}
\usetikzlibrary{calc,shapes}
\usepackage{wrapfig}
\usepackage{epsfig}  
\usepackage[margin=1.3in]{geometry}
\usepackage{mathtools}
\usepackage{color}	 %
\input{xy}
\xyoption{poly}
\xyoption{2cell}
\xyoption{all}
\newtheorem{thm}{Theorem}[section]
\newtheorem{prop}[thm]{Proposition}
\newtheorem{lemma}[thm]{Lemma}
\newtheorem{corollary}[thm]{Corollary}

\newtheorem*{thm*}{Theorem}

\usepackage{hyperref}
\theoremstyle{definition}
\newtheorem{definition}[thm]{Definition}
\theoremstyle{remark}
\newtheorem{remark}[thm]{Remark}
\newtheorem{example}[thm]{Example}
\numberwithin{equation}{section}
\newcommand{\cals}{\mathcal{S}}

\newcommand{\calc}{\mathcal{C}}

\newcommand{\za}{\alpha}
\newcommand{\zb}{\beta}
\newcommand{\zd}{\delta}

\newcommand{\ze}{\epsilon}
\newcommand{\zg}{\gamma}

\newcommand{\zs}{\sigma}

\newcommand{\zO}{\Omega}
\newcommand{\kb}{\Bbbk}

\newcommand{\Hom}{\textup{Hom}}
\newcommand{\add}{\textup{add}}
\newcommand{\cok}{\textup{cok}}
\newcommand{\im}{\textup{im}}

\newcommand{\rad}{\textup{rad}\,}

\newcommand{\Supp}{\textup{Supp\,}}
\newcommand{\Fac}{\textup{Fac\,}}

\newcommand{\Ext}{\textup{Ext}}
\newcommand{\End}{\textup{End}}

\newcommand{\cmp}{\textup{CMP}}
\newcommand{\scmp}{\underline{\textup{CMP}}}

\newcommand{\wt}{\textup{w}}
\begin{document}

\title[On syzygy categories over Iwanaga-Gorenstein algebras]{On syzygy categories over Iwanaga-Gorenstein algebras: Reduction, minimality and finiteness} 
\author{Ralf Schiffler}
%\thanks{}
\address{Department of Mathematics, University of Connecticut, Storrs, CT 06269-1009, USA} 
\email{schiffler@math.uconn.edu}

\author{Khrystyna Serhiyenko}
\thanks{The first author was supported by the NSF grants  DMS-2054561 and DMS-2348909. The second author was supported by the NSF grants DMS-2054255 and DMS-2451909. This work was supported by a grant from the Simons Foundation International [SFI-MPS-TSM-00013650, KS]}
\address{Department of Mathematics, University of Kentucky, Lexington, KY 40506-0027, USA}
\email{khrystyna.serhiyenko@uky.edu}

%\subjclass[2010]{Primary  13F60 %cluster algebras
%Secondary
%16G20} %Representations of quivers and partially ordered sets

\maketitle
\setcounter{tocdepth}{2}

\begin{abstract} 
We study 2-Calabi-Yau tilted algebras which are non-commutative Iwanaga-Gorenstein algebras of Gorenstein dimension 1.  In particular, we are interested in their syzygy categories or equivalently the stable categories of Cohen-Macauley modules $\underline{\text{CMP}}$.  First we show that if an algebra $A$ is Iwanaga-Gorenstein of Gorenstein dimension 1 then its stable category is generated under extensions by its radical $\text{rad}\,A$.  Next, for a 2-Calabi-Yau tilted algebra $A$ we provide an explicit relationship between the $\underline{\text{CMP}}$ category of $A$  and its quotient $A/Ae_iA$ by an ideal generated by an idempotent $e_i$.   Consequently, we obtain various equivalent characterizations of when the $\underline{\text{CMP}}$ category remains the same after passing to the quotient.  We also obtain applications to two classes of algebras that are CM finite, the dimer tree algebras and their skew group algebras. 
\end{abstract}

\section{Introduction}
 The study of syzygies was initiated in Hilbert's famous work in the 19th century \cite{Hilbert} and has played an important role in algebra every since. By definition, a syzygy is a submodule of a free module, and therefore syzygies make up the kernels in free resolutions and projective resolutions.

From the point of view of representation theory, it is interesting to study the category of all syzygies over an algebra, especially in the case where a good understanding of the whole module category is out of reach.

In this paper, we let $A$ be an Iwanaga-Gorenstein algebra of Gorenstein dimension at most 1. This means that the projective $A$-modules have injective dimension at most 1 and the injective $A$-modules have projective dimension at most 1.  In this situation, an $A$-module $M$ is a syzygy if and only if $M$ is  \emph{Cohen-Macaulay}, which means that $\Ext^1(M,A)=0$. We denote the category of all syzygies over $A$ by $\cmp A$.

The Cohen-Macaulay modules also form an important class of modules that is very well studied in commutative algebra, see \cite{E, LW, Y}. In the non-commutative setting of Iwanaga-Gorenstein algebras, these modules were promoted in the seminal work of Auslander and Reiten \cite{AR} and Buchweitz \cite{Bu}. An important problem is the classification of rings of finite Cohen-Macauley type, which is solved for hypersurface singularities and for normal Cohen-Macauley rings of Krull dimension two, in the commutative case \cite{AV,Au,BGS, Es, Kn}.
  For higher dimensions, as well as for non-commutative rings, the problem is open. 

The category $\cmp A$ is a Frobenius category and its stable category $\scmp A$ is a triangulated category whose shift is given by the inverse syzygy functor $\zO^{-1}$. It was shown in \cite{Bu} that $\scmp A$ is equivalent to the singularity category of the algebra $A$.

\smallskip

Our first main result shows that the radical $\rad A$ generates the Cohen-Macaulay category. We say a category $\calc$ is \emph{generated under extensions} by an object $X$ if $\calc$ is the smallest category containing $X$ that is closed under extensions and direct summands.

\begin{thm}
 \label{thmintro2} (Theorem \ref{prop:genrad})
 Let $A$ be an Iwanaga-Gorenstein algebra of Gorenstein dimension 1. Then $\scmp A$ is generated under extensions by the $A$-module $\rad A$.
\end{thm}

\smallskip
Next we restrict the generality of our setting and let $A$ be a 2-Calabi-Yau tilted algebra.  This means that $A$ is the endomorphism algebra of a cluster-tilting object in a 2-Calabi-Yau category. Examples of 2-Calabi-Yau tilted algebras are cluster-tilted algebras and Jacobian algebras of quivers with potential \cite{Amiot}. This is a large and well-studied class of algebras, see for example the surveys \cite{Assem, Reiten}. It was shown in \cite{KR} that $A$ is Iwanaga-Gorenstein of Gorenstein dimension at most 1 and that $\scmp A$ is a 3-Calabi-Yau category \cite{KR}.

In this paper, we study the behavior of the Cohen-Macaulay category under the reduction of the algebra obtained by removing a primitive idempotent, or equivalently, by removing a vertex from the quiver. 
More precisely, let $Q$ be the Gabriel quiver of $A$ and let $i$ be a vertex of $Q$. The constant path $e_i$ induces a primitive idempotent in $A$, which we also denote by $e_i$.  

Let $J=A e_i A$ be the two-sided ideal generated by $e_i$ and let $B=A/J$ be the quotient algebra. 
Then $B$ is also 2-Calabi-Yau tilted (Proposition~\ref{prop 22}) and the quiver of $B$ is obtained from $Q$ by removing the vertex $i$ and all arrows incident to it. We are studying the relation between $\cmp A $ and $\cmp B$.

The following theorem is our second main result. It gives a realization of $\cmp B$ inside $\cmp A$ that uses the ext-perpendicular category $J^\perp$ of $J$, whose objects are the Cohen-Macaulay modules that have no extension with  $J$.  Let $\Fac J$ denote the full subcategory of $\text{mod}\, A$ consisting of all modules that are quotients of sums of $J$. 

%\[J^\perp = \{ M \in \cmp\,A\mid \Ext^1_A(J,M)=\Ext^1_A (M,J)=0\}.\]

\begin{thm}
 \label{thmintro1}
(Theorem \ref{thm:eq})
 Let $A$ be a 2-Calabi-Yau tilted algebra, $J=Ae_iA$ and $B=A/J.$ Then there are equivalences of categories
 \[ {J^\perp/(\Fac J)} \cong \cmp\,B \qquad and \qquad\underline{J^\perp/(\Fac J)} \cong \scmp\,B,\]
 where $(\Fac J)$ is the ideal of $\cmp\, A$ of all morphisms factoring through $\Fac J \cap \cmp\, A$.
 \end{thm}

We also give an explicit construction of the functors that realize these equivalences. We remark here that the definition of the functor uses the fact that the category $\scmp A$ is 3-Calabi-Yau, which is why we require $A$ to be 2-Calabi-Yau tilted.

 As a consequence of these two theorems, we obtain the following five characterizations for the equivalence of the syzygy categories of $A$ and $B$.

\begin{corollary}\label{corintro}
(Corollary \ref{cor})
 Let $A$ be a 2-Calabi-Yau tilted algebra, and let $m_{ji}%=\dim \Hom_A(P_A(i), P_A(j))
 =\dim \,e_j Ae_i$. Then the following statements are equivalent.\smallskip
\begin{enumerate}%\setlength\itemsep{1em}
\item [(a)] $\bar F: \scmp\,A \to \scmp\,B$ is an equivalence. \smallskip
\item [(b)] $J$ is a projective $A$-module and $m_{ii}=1$.\smallskip
\item [(c)] $J\cong \oplus_{j\in Q_0} P_A(i)^{m_{ji}}$.\smallskip
\item [(d)] For  all $j\in Q_0$, the module $P_A(i)^{m_{ji}}$ is a submodule of $ P_A(j)$.\smallskip

\item [(e)] The radial $\rad\, P(i)$ is generated by the radicals  $ \rad P(j)$, with ${j\not=i}$.\smallskip
\item [(f)] There exists a short exact sequence $0\to \rad\, P(i) \oplus Y\to P \to X\to 0$ where  $Y\in\cmp A$, $P$ is a projective $A$-module, and $X$ is an $A$-module that is not supported at $i$.
\end{enumerate}
\end{corollary}

\begin{remark}
 The construction in Theorem~\ref{thmintro1} is reminiscent of Iyama-Yoshino \cite{IY} reduction in two ways. 

(i) On the level of 2-Calabi-Yau tilted algebras and their cluster categories, Iyama-Yoshino reduction describes a relation between the 1-cluster categories $\mathcal{C}_A$ of $A$ and $\mathcal{C}_B$ of $B$ via the perpendicular category $(P_A(i)[1])^\perp$ of the shift of the indecomposable projective at vertex $i$ inside $\mathcal{C}_A$. 

(ii) Our construction $\underline{J^\perp/(\Fac J)}$ in Theorem \ref{thmintro1} is an analog of the Iyama-Yoshino reduction (because $J$ is rigid, see Lemma~\ref{lem:J} (g)) of the 2-cluster category $\scmp A$. This category is very different from the reduction $\mathcal{C}_B$ in (i) above. 

One may think of   Theorem~\ref{thmintro1} as relating these two reductions by proving that $\underline{J^\perp/(\Fac J)}$ of (ii) is equivalent to the stable Cohen-Macaulay category of the reduced algebra $B$ in (i).
\end{remark}
\smallskip

\subsection{Finite CM-type}
We then turn our attention to 2-Calabi-Yau tilted algebras that have finite Cohen-Macaulay categories. Such an algebra is said to be of \emph{finite CM-type}. It follows from \cite{Amiotthesis} that
in this case the stable Cohen-Macaulay category is  contained in a 2-cluster category of Dynkin type $\mathbb{A,D}$ or $\mathbb{E}$ (in fact, the 2-cluster category is a finite covering of the stable CM category). 
We refer to this Dynkin type as the \emph{CM-type} of the algebra. 

One motivating goal for our research is to classify 2-Calabi-Yau tilted algebras of finite CM-type. Corollary \ref{corintro} is a useful tool for this objective. For an example consider the algebras shown in Figure~\ref{fig ex 1}. Each is the Jacobian algebra of the shown quiver with respect to the potential given by the sum of the chordless cycles. The quiver of the algebra $A$ is the shown in the leftmost picture. The other algebras are obtained by reductions of $A$. Corollary~\ref{corintro} implies that all these algebras have isomorphic CMP categories. In fact, they are all of type $\mathbb{A}_2$.

In our earlier work \cite{SS3,SS4,SS5} we introduced and studied important classes of  algebras of finite CM-type which we call \emph{dimer tree algebras} as well as their skew group algebras. The dimer tree algebras are of CM-type $\mathbb{A}$ and their skew group algebras have CM-type $\mathbb{D}$. The algebra  $A$ in Figure~\ref{fig ex 1} is not a dimer tree algebra, but all of its reductions are dimer tree algebras.

Let $A$ be a dimer tree algebra with quiver $Q$. As before, let $i$ be a vertex in $Q$, $J=A e_i A$ and $B=A/J$. 
In Theorem \ref{thm A reduction}, we characterize combinatorially when the algebra $A$ and $B$ have the same CM-type in terms of the local neighborhood of the vertex $i$ in the quiver of $A$. 
%
%\begin{thm} \label{thmintro3}
%(Theorem \ref{thm A reduction})
% With the notation above, the following statements are equivalent. 
% 
%\begin{enumerate}
%\item[(a)] $\cmp A \cong \cmp B$
% 
%\item[(b)] There exists a 3-cycle in $Q$ of the form $\xymatrix{h\ar[r]^\zb&i\ar[r]^\zg&j\ar@/^10pt/[ll]^\zd}$ with $\zb$ and $\zg$ boundary arrows and $\zd$ an interior arrow such that $\overline {\wt}(\zb)=1$ and $\wt(\zg)=1$. 
% \end{enumerate} 
% \footnote{remove theorem or define weight and coweight}
%
%\end{thm}

This characterization is useful, for example, to find minimal representatives of a given CM-type. We say an algebra $A$ is \emph{CM-minimal} if the CM-type of $A$ is different from the CM-type of its reduction $A/A e_i A$, for all vertices $i$.

As an application, we obtain the following result on the relation between the minimality of dimer tree algebras their skew group algebras.
\begin{thm}
 \label{thmintro3} Let $A$ be a dimer tree algebra and $AG$ a corresponding skew group algebra. Then
$A$ is CM-minimal if and only if $AG$ is. 
\end{thm}

\smallskip

\subsubsection*{Related work}

The syzygy categories of finite dimensional algebras have been studied in different settings by several research groups. The work \cite{CGL} focused on cluster-tilted algebras of Dynkin type, and \cite{Lu} extended their results to polygon-tree algebras. The case of gentle algebras was studied in \cite{Kalck} and generalized to skew-gentle algebras in \cite{CL}. Further results can be found in \cite{AIR,Chen,C2, GES, LZ,Sh}.

\smallskip

The paper is organized as follows.  Theorem~\ref{thmintro2} is proved in section~\ref{sect 2}, and Theorem~\ref{thmintro1} and Corollary~\ref{corintro} in section~\ref{sect 3}. The applications to finite CM-type, including Theorem~\ref{thmintro3}, are presented in section~\ref{sect 4}.

\begin{figure}
\begin{minipage}{.2\textwidth}
\hspace*{-.4cm}$A:$\xymatrix{
&4\ar[d]&1\ar[l]\ar[d]\\
6\ar[ur]&3\ar[d]\ar[ur]&2\ar[l]\\
&5\ar[ul]\ar[ur]} 
\end{minipage}
\begin{minipage}{.3\textwidth}
\hspace*{.3cm}$A/Ae_1A:$\xymatrix{
&4\ar[d] \\
6\ar[ur]&3\ar[d]&2\ar[l]\\
&5\ar[ul]\ar[ur]} 
\vspace{1cm}
\hspace*{.3cm}$A/Ae_2A:$\xymatrix{
&4\ar[d]&1\ar[l]\\
6\ar[ur]&3\ar[d]\ar[ur]\\
&5\ar[ul]} 
\end{minipage}
\begin{minipage}{.3\textwidth}
\hspace*{.5cm}$A/Ae_6A:$\xymatrix{
&4\ar[d]&1\ar[l]\ar[d]\\
&3\ar[d]\ar[ur]&2\ar[l]\\
&5\ar[ur]
}
\vspace{1cm}
\hspace*{.5cm}$A/A(e_1+e_2)A:$\xymatrix{
&4\ar[d]\\
 6\ar[ur]&3\ar[d]\\
&5\ar[ul]} 
\end{minipage}
\caption{A dimer algebra $A$ given by the quiver and its four reductions at the idempotents $e_1, e_2, e_6, e_1+e_2$ respectively. All of these algebras are of CM type $\mathbb{A}_2$, and the algebras in the right column are CM minimal.}
\label{fig ex 1}
\end{figure}
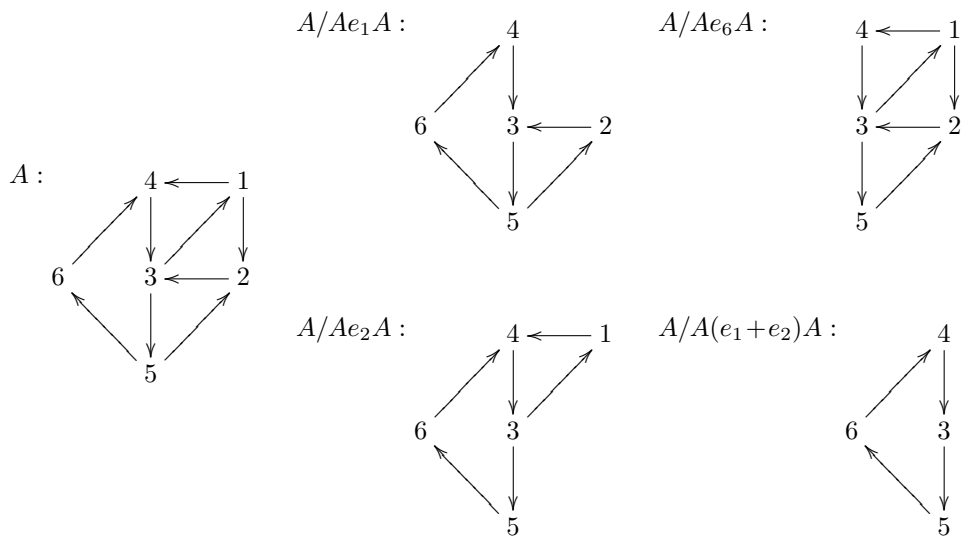

%%%%%%%%%%%%%%%%%%%%%%%%%%%%%%%%%%%%%%%%%%%%%%%%%%%%%%%%%%
%%
%%         SECTION
%%
%%%%%%%%%%%%%%%%%%%%%%%%%%%%%%%%%%%%%%%%%%%%%%%%%%%%%%%%%

\section{Syzygies of 1-Iwanaga Gorenstein algebras are generated by radicals}\label{sect 2}
 Throughout the paper, algebras will be finite dimensional algebras over an algebraically closed field $\kb$. We denote the category of finitely generated right $A$-modules by $\textup{mod}\,A$.   For each vertex $i$ in the Gabriel quiver of the algebra, we denote by $P(i),I(i)$ and $S(i)$ the corresponding indecomposable projective, injective, and simple module, respectively. The radical $\rad P(i)$ is the kernel of the projective cover $P(i)\to S(i)$. 
For further information on representation theory, we refer to \cite{ASS,S2}.

\smallskip
Throughout this section, let $A$ be an Iwanaga-Gorenstein algebra of Gorenstein dimension at most 1, which means that the injective dimension of every projective $A$-module is at most one, and the projective dimension of every injective $A$-module is at most one.    Let $\cmp\,A$ denote the category of \emph{Cohen-Macaulay} modules over $A$, that is, 
\[\cmp\,A=\{ M \in \textup{mod}\,A\mid \Ext^1_A(M, A)=0\},\]
and let $\scmp\,A$ denote the associated stable category.  By \cite{Bu}, the category $\cmp\,A$ is a Frobenius category whose projective-injective objects are the projective $A$-modules, and its stable category  $\scmp\,A$ is a triangulated category with inverse shift given by the syzygy functor $\Omega$ . Moreover,  the category $\scmp\,A$ also coincides with the category of non-projective syzygies over $A$.

For a collection of objects $S$ in $\scmp A$, let $\overline{S}$ denote the smallest  subcategory of $\scmp A$ that is  extension closed, closed under direct summands, and contains $S$.  We say that an object $M$ of $\scmp A$ is \emph{generated by $S$} if $M$ lies in $\add\, \overline{S}$.  

\begin{lemma}\label{genrad}
 Let $I$ be a subset of the vertices of the quiver of $A$. 
 Then the following are equivalent  for $M\in\scmp A$. 

(a) $M$ is generated by the radicals $\rad P(i), i\in I$;

(b) there exists a short exact sequence
 \[0\to Y_M
 \to P_M \to X_M\to 0,\] 
 such that $Y_M\in \cmp A$ that contains $M$ as a summand,
 $P_M$ is projective, and $X_M$ is supported only at vertices of $I$.
 \end{lemma}

\begin{proof} 
Clearly, if $M=\rad P(i)$ with $i\in I$ then (b) holds with  $Y_M=\rad P(i)$,  
$P_M=P(i)$, and $X_M=S(i)$. Thus to show that (a) implies (b), it suffices to show that condition (b) is preserved under extensions. 
Let $L,N\in \scmp A$ be two objects that satisfy condition (b). Thus there are short exact sequences
\begin{equation}
 \label{eq K1} \xymatrix{0\ar[r]&Y_L  \ar[r]^f&P_L\ar[r]&X_L\ar[r]&0}
\end{equation}
\begin{equation}
 \label{eq K2} \xymatrix{0\ar[r]&Y_N  \ar[r]^g&P_N\ar[r]&X_N\ar[r]&0}
\end{equation}
where $Y_L, Y_N \in \cmp A$ that contain $L, N$ as summands respectively, 
$P_L, P_N$ are projective and $X_L,X_N$ are only supported at vertices of $I$.
Suppose that $L$ and $N$ have an extension

 %\label{eq K3}
\[\xymatrix{0\ar[r]&L\ar[r]&M\ar[r]&N\ar[r]&0}.\]

We want to show that $M$ also satisfies condition (b).  
 Then in particular, $Y_L$, and $Y_N$ also have an extension 
\begin{equation}
 \label{eq K3}
 \xymatrix{0\ar[r]&Y_L\ar[r]&Y_M\ar[r]&Y_N\ar[r]&0}
\end{equation}
such that $M$ is a summand of $Y_M\in \cmp A$.  In particular, $Y_M \cong M \oplus Y_L/L\oplus Y_N/N$.

%Composing the map $f$ in the sequence (\ref{eq K1}) with the inclusions $\rad P(i)\subset P(i) $,  we obtain the following short exact sequence.
%\begin{equation}
% \label{eq K3} \xymatrix{0\ar[r]&L\oplus P_L\ar[r]&\mathop{\oplus}\limits_{i\in I} (P(i))^{\ell_i}\ar[r]&\overline{X}_L\ar[r]&0}
%\end{equation}
%for some $\overline{X}_L$. Note that the dimension vectors satisfy the identity $\underline \dim \overline{X}_L -\underline \dim X_L= \sum_{i\in I} e_i^{\ell _i}$. In particular,
% $\overline{X}_L$ too is only supported at vertices of $I$.
 
Consider the following commutative diagram, 
 \[\xymatrix{
 &0\ar[d]&0\ar[d] \\
 0\ar[r]&Y_L \ar[r]\ar[d]_f&
 Y_M
\ar[r]\ar[d]&Y_N \ar[r]\ar[d]^\cong&0\\
 0\ar[r]& P_L
    \ar[r]\ar[d]&E\ar[r]\ar[d]& Y_N \ar[r]&0\\
   &X_L  \ar[r]^\cong\ar[d]&X_L \ar[d]\\
  &0&0}
 \] 
%
% \[\xymatrix@!@R0pt{
% &0\ar[d]&0\ar[d] \\
% 0\ar[r]&L\oplus P_L\ar[r]\ar[d]&M\oplus P_L\ar[r]\ar[d]&N\ar[r]\ar[d]^\cong&0\\
%  0\ar[r]&\mathop{\oplus}\limits_{i\in I} P(i)^{\ell_i}\ar[r]\ar[d]&E\ar[r]\ar[d]&N\ar[r]&0\\
%  &\overline{X}_L\ar[r]^\cong\ar[d]&\overline{X}_L\ar[d]\\
%  &0&0
%  }
% \]
where the top row is the sequence (\ref{eq K3}), the top left square is a push out diagram inducing the isomorphism in the third column, and the isomorphism in the third row follows from the snake lemma.
Since $Y_N \in \cmp A$, it has no extensions with projectives. Therefore the sequence in the middle  row splits and hence $E\cong P_L\oplus Y_N$. Thus we obtain the following commutative diagram
\begin{equation}\label{D2}
\xymatrix{
0\ar[r]&Y_M\ \ar[d]^\cong  
\ar[r]^-{\zb}
& P_L\oplus Y_N
\ar[r]
\ar[d]^{\left[ \begin{smallmatrix}1&0\\0&g\end{smallmatrix}\right]}
& X_L\ar[r]\ar[d]^\zg&0
\\
0\ar[r]
&Y_M   \ar[r]^(.33)\za
&P_L\oplus P_N \ar[r] 
& \cok \,\za\ar[r]&0
\\
}
\end{equation}
where the top row is the center column of the previous diagram, the map $g$ is from the sequence (\ref{eq K2}), and 
$\za$ is the composition of $\beta$ with $\left[ \begin{smallmatrix}1&0\\0&g\end{smallmatrix}\right]$. The map $\zg$ exists by the universal property of the cokernel of $\beta$. 
Using the fact that  $g$ is injective with cokernel $X_N$, the snake lemma then implies that $\zg $ is injective and we have  the short exact sequence
\[\xymatrix{0\ar[r]&X_L\ar[r]^\zg&\cok\, \za \ar[r]& X_N\ar[r]&0}.\]
Note that since $X_L, X_N$ are supported only at vertices of $I$ then the same holds for $\cok\, \za$.  Thus, $M$ satisfies condition (b) with $P_M=P_L\oplus P_N$ and $X_M=\cok\, \za$, see the short exact sequence in the second row of diagram \eqref{D2}.  This completes the proof that (a) implies (b).

To show the reverse implication, suppose that $M\in \scmp\,A$ is nonzero.  First we claim that if $M$ is generated by the radicals $\rad P(i), i\in I$ then so is a submodule of $M$ obtained by removing a simple $S(i)$ with $i\in I$ from the top of $M$. This means there exists a short exact sequence of the form ${0\to M'\to M\to S(i)\to 0 }$. We prove the claim as follows. Given $M$, consider the following commutative diagram
\begin{equation}
\label{D3}
\xymatrix{
&0\ar[d] &0\ar[d]\\
0\ar[r] & \Omega M \ar[r]^f \ar[d]_{f_i}& P \ar[r] \ar[d]& M \ar[r] \ar[d]^{\cong}& 0\\
0 \ar[r] & L_i \ar[d]\ar[r] & E \ar[r] \ar[d]& M\ar[r] & 0\\
&M_i\ar[r]^{\cong}\ar[d] & M_i\ar[d]\\
&0&0}
\end{equation}
where the map $P\to M$ in the first row is a minimal projective cover of $M$ with kernel $\Omega M$.  Let $S(i)$ be a direct summand of $\textup{top}\, M$. Then $P\cong P'\oplus P(i)$ with $P'$ projective.  Define $L_i = P'\oplus \rad\, P(i)$.  In particular, $L_i$ is obtained from $P$ by replacing a single indecomposable projective $P(i)$ by its radical. Then $f$ factors through the injective map $f_i: \zO M\to L_i$ and the top left square of the diagram is the push out of the first row along $f_i$.  Observe that the cokernel of $f_i$ is the module $M_i$ obtained from $M$ by removing $S(i)$ from the top.  In particular, there exists a short exact sequence
\[0 \to M_i \to M \to S(i) \to 0.\]
Since $M_i$ is a submodule $M$ then $M_i\in \cmp\,A$.  In particular, $M_i$ has no extensions with projectives and the sequence in the middle column of \eqref{D3} splits.  Thus, $E \cong P\oplus M_i$ and the second row of  \eqref{D3} becomes the following short exact sequence. 
\[0\to L_i \to P\oplus M_i\to M \to 0\]
Since $M\in\cmp A$  has no extensions with projectives, we can remove the projective summand $P'$ from both $L_i$ and $P$ from the above sequence to obtain a short exact sequence as follows. 
\[ 0 \to \rad \,P(i)\to  P(i)\oplus  M_i \to M \to 0\]
This shows the desired claim that if $M$ satisfies condition (a) then so does $M_i$ for $i\in I$.  
Recall that we want to show that (b) implies (a).  We do so now by induction on the number of summands in $P_M$.  If $M$ satisfies (b) with $P_M=P(i)$ with $i\in I$, then $M\subset \rad P(i)$ and thus $M$ can be obtained from $\rad\,P(i)$ by successively removing simple tops supported at vertices of $I$.  More precisely, there exists a sequence of modules for some $t\geq 0$
\[M=M_{i_0}\subset M_{i_1}\subset \dots \subset M_{i_t}=\rad P(i)\]
such that the successive quotients $M_{i_j}/M_{i_{j-1}}\cong S(i_j)$ where $i_j\in \Supp X_M\subset I$.  Since $\rad P(i)$ satisfies (a), then recursively applying the claim shows that $M_{i_{t-1}}, \dots, M_{i_0}$ also satisfy (a).  This shows the base case that if $M$ satisfies (b) with $P_M$ indecomposable then $M$ satisfies (a).  

Now, suppose that (b) implies (a) for all modules $N\in \scmp\,A $ such that $P_N$ consists of at most $k\geq 1$ indecomposable summands.  Let $M$ satisfy (b) where $P_M$ consists of $k+1$ indecomposable summands.  Then consider the following diagram
\[
\xymatrix{
&0\ar[d] & 0 \ar[d] & 0\ar[d]\\
0\ar[r] & M'' \ar[r]^{f''} \ar[d]& P(i) \ar[r] \ar[d]& X'' \ar[r] \ar[d]& 0\\
0\ar[r] &  Y_M \ar[r]^f \ar[d]^\pi& P_M \ar[r] \ar[d]^{\pi_P}& X_M \ar[r] \ar[d]^{\pi_X}& 0\\
0\ar[r] & M' \ar[r]^{f'} \ar[d]& P_{M'} \ar[r] \ar[d] & X' \ar[r] \ar[d]& 0\\
& 0 & 0 & 0
}
\]
where the sequence for $M$ appears in the second row.  Here $P_M \cong P_{M'}\oplus P(i)$ for some indecomposable projective $P(i)$ with $i\in I$, which gives the split short exact sequence in the second column of the diagram.  Let $M'=\im\, \pi_P f$, then there exists a surjective map $\pi: Y_M\to M'$ such that the bottom left square commutes.  Let $X'=\cok\, f'$, then the map $\pi_X$ exists by the universal property of the cokernel of $f$.  Note that since $\pi, \pi_P$ are surjective then so is $\pi_X$.  Then the first row of the diagram is obtained by taking kernels of $\pi, \pi_P, \pi_X$.  Note that since $X_M$ is supported only at vertices of $I$, then so are $X', X''$.  In particular, the modules $M', M''$ satisfy condition (b) with  $Y_{M'}=M', Y_{M''}=M''$ and $P_{M'}, P_{M''}=P(i)$ consisting of at most $k$ indecomposable summands.  Hence, by the induction assumption $M', M''$ are generated by radicals $\rad P(i), i\in I$.  Then the first column of the diagram implies that $M$,  which is a summand of $Y_M$, is also generated by these radicals.  Thus $M$ satisfies (a).  This completes the proof that (b) implies (a). 
\end{proof}

\begin{thm}\label{prop:genrad} Let $A$ be an Iwanaga-Gorenstein algebra of Gorenstein dimension at most 1. Then the radicals of the projective $A$-modules generate $\scmp\,A$. 
\end{thm}

\begin{proof}
Take $I=Q_0$, then any $M\in\scmp\,A$ is a submodule of a projective, so it satisfies condition (b) of Lemma~\ref{genrad}.  Then the lemma implies that $M$  is generated by the radicals. 
\end{proof}

\begin{remark}
The following alternative proof of Theorem~\ref{prop:genrad} was communicated to us by Martin Kalck and independently by Xiaofa Chen.  Observe that $\text{mod}\, A$ is generated by the simple modules $S(i)$ under extensions. By the Horseshoe Lemma any short exact sequence \[0 \to X \to Y \to Z \to 0\] 
in $\text{mod}\,A$ yields a short exact sequence in $\cmp\,A = \Omega (\text{mod}\,A)$ 
\[0 \to \Omega X \to \Omega Y \oplus P \to \Omega Z \to 0\] 
for some projective $A$-module $P$.  This implies that $\cmp\,A$ is generated by the $\Omega S(i)\cong \rad\, P(i)$.  Moreover, the same argument can be used to show that (b) implies (a) in Lemma~\ref{genrad}.  Note that Lemma~\ref{genrad} is needed in the proof of Corollary~\ref{corintro}.
\end{remark}

\section{Reduction of 2-Calabi-Yau tilted algebra}
\label{sect 3}
 Let  $A$ be a (finite-dimensional) 2-Calabi-Yau tilted algebra over an algebraically closed field $\kb$. Let $A=\kb Q/I$ be its description as the quotient of a path algebra with quiver $Q$ and admissible ideal $I$. The notation $Q_0$ and $ Q_1$ stands for the vertices and arrows of $Q$, respectively. 
%We denote the category of finitely generated right $A$-modules by $\textup{mod}\,A$, and the indecomposable projective, injective and simple module at the vertex $i\in Q_0$ by $P(i), I(i), $ and $S(i)$, respectively.
It was shown in \cite{KR} that $A$ is Iwanaga-Gorenstein of Gorenstein dimension 1 and its stable syzygy category $\scmp\,A$ is a 3-Calabi-Yau triangulated category with inverse shift given by $\Omega$.
% Let $\cmp\,A$ denote the category of \emph{Cohen-Macaulay} modules over $A$, that is  
%\[\cmp\,A=\{ M \in \textup{mod}\,A\mid \Ext^1_A(M, A)=0\},\]
%and let $\scmp\,A$ denote the associated stable category.  The category $\cmp\,A$ is a Frobenius category whose projective-injective objects are the projective $A$-modules.  Moreover,  It follows from \cite{} that $\scmp\,A$ also coincides with the category of non-projective syzygies over $A$. 

Let $e_i$ be a primitive idempotent at vertex $i$, and we denote by $J=Ae_iA$ the two-sided ideal of $A$ generated by $e_i$.    Note that $J$ is also a module over $A$.   Now, consider the quotient algebra $B=A/J$.  Observe that every $B$-module is also an $A$-module where the elements of $J$ act trivially.  Moreover, there exists a short exact sequence 
\begin{equation}\label{eq1}
0\to J \to A \to B\to 0.
\end{equation}

Note that the quiver of $B$ is obtained from that of $A$ by removing the vertex $i$ and the relations of $B$ are obtained from the relations of $A$ by replacing by zero every monomial that corresponds to a path through the vertex $i$.
An example \ref{ex1} is given at the end of this section.

We begin with a preliminary lemma about the properties of $J$. This result holds in the more general setting of 1-Iwanaga-Gorenstein algebras.

\begin{lemma}\label{lem:J}
Let $A$ be a be an Iwanaga-Gorenstein algebra of Gorenstein dimension 1, and let $B = A/J$, where $J=Ae_iA$.  Then the right $A$-module $J$ satisfies the following.
\begin{itemize}
\item[(a)] $J\in \cmp\,A$,
\item[(b)]  $J\cong \sum_w wA$, where the sum is over all (residue classes of) paths $w\in A$ that end in i,
\item[(c)] $J\cong\oplus_{j\in Q_0} e_jJ,$
\item[(d)] $P_A(i)\in \add\, J$,
\item[(e)] $\textup{top}\,J \in \textup{add}\,S(i)$,
\item[(f)] $\Hom_A(J, M_B)=0$ for every $M_B\in \textup{mod}\,B$, 
\item[(g)] $\Ext^1_A(J,J)=0$.
\end{itemize}
\end{lemma}

\begin{proof}
By \eqref{eq1} the module $J$ is a submodule of a projective $A$-module $A_A$, which shows (a).  

Let $\mathcal{B}(A)$ be a basis of  $A$ whose elements are paths in $Q$. The elements of $J$ are $k$-linear combinations of the elements of the form $we_ia$ with $w\in \mathcal{B}(A), a\in A$, and $we_i$ is zero unless $w$ ends in $i$. Thus 
$J\cong \sum_w wA$ as $k$-vector spaces. Since the right $A$-module structures respect this decomposition, this proves (b).

(c) holds because the identity $1_A$ is equal to the sum of the idempotents $1_A=\sum_{i\in Q_0} e_i$,  and 
(d) follows from (b) when we take $w=e_i$.

%To prove (b), multiply the ideal $J$ on the right by $e_i$ to obtain $Je_i=Ae_iAe_i=Ae_i=P_A(i)$.  Hence, $P_A(i)$ is a summand of $J$.

%Now, multiplying \eqref{eq1} on the right by a primitive idempotent $e_j$ with $j\not=i$, we get that the projective $B$-module $P_B(j)$ is a quotient of the corresponding projective $A$-module $P_A(j)$.  Since $J$ consists of all paths that go through vertex $i$, it follows that $P_B(j)$ is obtained from $P_A(j)$ by removing bases elements corresponding to paths that go through vertex $i$. Hence the top of $P_A(j)/P_B(j) \in \add\,S(i)$, which shows (c). 

To show (e), observe that for every path $w$ ending in $i$ the map $e_i a\mapsto wa$ gives a surjective morphism $e_iA\to wA$ in $\textup{mod}\,A$. Hence $\textup{top} \,wA=S(i)$, and  part (b) implies that $\textup{top} \, J\in \textup{add}\,S(i)$.

Part (f) follows from (e) since no $B$-module is supported at vertex $i$. 

To show the last part, apply $\Hom_A(J, -)$ to \eqref{eq1} to obtain the following exact sequence. 
\[ \Hom_A(J, B) \to \Ext^1_A(J,J) \to \Ext^1_A(J, A)\]
Note that the first term is zero by (f) and the last term is zero by (a), which implies that $\Ext^1_A(J,J)=0$. 
\end{proof}

The next result shows that the reduction preserves the 2-Calabi-Yau tilted property.

\begin{prop}\label{prop 22}
If the algebra $A$ is 2-Calabi-Yau tilted then $B=A/J$ is also 2-Calabi-Yau tilted. 
\end{prop}

\begin{proof}
Since $A$ is a 2-Calabi-Yau tilted algebra, then by definition $A = \End_{\mathcal{C}}(T)$ where $\mathcal{C}$ is a (hom-finite, Krull-Schmidt) 2-Calabi-Yau category and $T$ is a basic cluster-tilting object in $\mathcal{C}$.  Let $T_i$ be an indecomposable direct summand of $T$. Consider the full subcategory of $\mathcal{C}$ 
\[T_i^\perp = \{X\in \mathcal{C} \mid \Ext^1_{\mathcal{C}}(T_i, X)=0\}\]
consisting of objects that have no extensions with $T_i$.  By \cite[Theorem 4.7]{IY}, the category $\mathcal{C'} = T_i^\perp/(T_i)$, obtained from $T_i^\perp$ by factoring by the ideal of morphisms that factor through $\add\,T_i$, is again a 2-Calabi-Yau category.   Note that the cluster tilting object $T\in T_i^\perp$.  Moreover,  by \cite[Theorem 4.9]{IY}, the object $T\in \mathcal{C}$ yields a cluster tilting object $T/T_i \in \mathcal{C'}$.  Then we have the following sequence of isomorphisms 
\[ \End_{\mathcal{C'}}(T/T_i) \cong \End_{\mathcal{C}}(T/T_i)/(T_i)\cong \End_{\mathcal{C}}(T)/(T_i) = A/Ae_iA = B.\]
In particular, it shows that $B$ is an endomorphism algebra of a cluster tilting object in a 2-Calabi-Yau category.  Therefore, $B$ is a 2-Calabi-Yau tilted algebra. 
\end{proof}

We consider the following perpendicular subcategory of $J$ inside $\cmp\,A$
\[J^\perp = \{ X \in \cmp\,A\mid \Ext^1_A(J,X)=\Ext^1_A (X,J)=0\}.\]

\begin{remark}
By the 3-Calabi-Yau property of $\scmp\, A$ we also have 
\[J^\perp = \{ X \in \cmp\,A\mid \Ext^1_A(J,X)=\Ext^2_A (J,X)=0\}.\] 
\end{remark}

\begin{prop}
 \label{new prop 1.3} 
 With the notation above, we have
 \begin{itemize}
 \item[(a)] $A_A\in J^\perp.$
 \item[(b)] $J\in J^\perp$.
\end{itemize}
\end{prop}
\begin{proof}
 (a) follows from part (a) of Lemma~\ref{lem:J} and (b) from part (g) of the same lemma.
\end{proof}

\begin{definition}\label{def:F}
We define a functor 
\[F: J^\perp \to \cmp\,B\] 
 as follows.  Given an object $X\in J^\perp$ set $F(X)= X/\text{im}\,f_X$, where $f_X: J_X\to X$ is an $\add\,J$ approximation of $X$.  Note that the existence of this approximation is given by the fact that $J$ has only a finite number of indecomposable direct summands, by Lemma~\ref{lem:J}(b). Moreover, given a morphism $g: M\to N$ in $J^\perp$ the morphism $F(g)$ is defined as the unique map such that the following diagram commutes. 
\[\xymatrix{M \ar[r] \ar[d]^g& F(M) \ar[d]^{F(g)} \\
N \ar[r] & F(N)}
\]
\end{definition}
 \begin{remark} \label{rem ses}
 In particular, we have the following short exact sequence in $\textup{mod}\,A$,
 \[ \xymatrix{0\ar[r]& \im f_X \ar[r] &X\ar[r]&F(X)\ar[r]&0}\]
 with $f_X\colon J_X\to X$ an $\add \,J$ approximation.
\end{remark}

\begin{prop}
The functor $F$ is well-defined. 
\end{prop}

\begin{proof}
First, we show that $F$ is well-defined on morphisms, meaning that there exists $F(g)$ as in Definition~\ref{def:F}.
Consider the diagram below, 
\[\xymatrix{J_M \ar[r]^{f_M} \ar@{..>}[d]^{u}& M \ar[r] \ar[d]^g& F(M) \ar[r] \ar@{..>}[d]^{F(g)}& 0 \\
J_N \ar[r]^{f_N} & N \ar[r] & F(N) \ar[r] & 0}
\]
where the rows come from the $\add\,J$ approximations of $M,N$ respectively and are exact.  Since $f_N$ is an approximation of $N$ by $\add\,J$ then there exists a map $u: J_M\to J_N$ such that $f_Nu= gf_M$.  Now, since the first square commutes and the rows are exact, the universal property of $\text{cok}\,f_M$ yields a unique map $F(g)$ such that the diagram commutes.

Now, we show that $F$ is well-defined on the objects.  Observe that since $P_A(i)$ is a summand of $J$ by Lemma~\ref{lem:J}(d), then an $\add\,J$ approximation of any $A$-module $X$ will have an image that contains the vector space of $X$ at vertex $i$.  Thus, $F(X)$, which is defined as the cokernel of such an approximation, is not supported at vertex $i$.  This shows that $F(X)\in \text{mod}\,B$, and it remains to prove that $F(X)\in \cmp\,B$. 

Let $X\in J^\perp$, and apply $\Hom_A(X, - )$ to the short exact sequence \eqref{eq1} to obtain a long exact sequence as follows. 
\[\Ext^1_A(X, A) \to \Ext^1_A(X, B) \to \Ext^2_A(X, J)\]
Observe that the first term is zero because $X\in J^\perp \subset \cmp\,A$.  On the other hand, the last term  $\Ext^2_A(X, J)$ is isomorphic to $D\Ext^1_A(J, X)$ by the the 3-Calabi-Yau property of $\cmp\,A$, and this last term is zero, because  $X\in J^\perp$.   This shows that $\Ext^1_A(X, B)=0$.  

Now, consider the short exact sequence 
\[0\to \text{im}\,f_X\to X \to F(X)\to 0\]
in $\text{mod}\,A$ where $f_X$ is the $\add\,J$ approximation of $X$.   Applying $\Hom_A(-,B)$ yields an exact sequence 
\[\Hom_A(\text{im}\,f_X, B)\to \Ext^1_A(F(X),B)\to \Ext^1_A(X,B).\]
The last term is zero by the above.  Now, consider the first term.  By definition the module $\text{im}\,f_X$ is a quotient of a module in $\add\,J$.  Thus, by 
Lemma~\ref{lem:J}(e)
 we have that $\text{top}\, \text{im}\,f_X \in \add\,S(i)$.  Since the module $B$ is not supported at vertex $i$, it follows that $\Hom_A(\text{im}\,f_X, B)=0$.  Therefore, $\Ext^1_A(F(X),B)=0$ which implies that $\Ext^1_B(F(X),B)=0$.  This shows that $F(X)\in \cmp\,B$.  
\end{proof}

\smallskip

Our next goal is to show that the functor $F$ is full and dense. First we need a few preparatory lemmata.

\begin{lemma}\label{lem:Fproj} On projective $A$-modules the functor $F$ acts as follows.
\begin{enumerate}
\item[(a)] $F(P_A(j))=P_B(j)$ if $j\ne i$; 
\item[(b)] $F(P_A(i))=0$.
\end{enumerate}
\end{lemma}
\begin{proof}
(a) follows by multiplying (\ref{eq1}) by $e_j$ on the left.
%\[\xymatrix@C20pt{0\ar[r]& e_j \,J  \ar[r]& P_A(j) \ar[r]& P_B(j)\ar[r]&0}\]
To show (b), observe that part (b) of
Proposition~\ref{new prop 1.3} implies that the add $J$ approximation of 
  $P_A(i) $ is the identity map, and hence $F(P_A(i))=0.$
\end{proof}

 \begin{lemma}
 \label{lem largest quotient} Let $X\in J^\perp$, and let $M$ be a $B$-module. Then every morphism $g\in \Hom_A(X,M)$ factors through $F(X)$. In particular, $F(X)$ is the largest quotient of $X$ in $\textup{mod}\,B$. 
\end{lemma}
\begin{proof}
 The composition $g\circ f_X\colon J_X\to M$ is zero, because $\textup{top}\, J_X \in \add \, S(i)$ and the $B$-module $M$ is not supported at $i$. Therefore the statement follows from the universal property of the cokernel of $f_X$ and the short exact sequence in Remark~\ref{rem ses}.
\end{proof}

\begin{lemma}\label{lem:ses}
Given $Y\in \cmp\,B$ there exists a short exact sequence 
\[\xymatrix{0 \ar[r]& J_X\ar[r]^{f_X} & X \ar[r]& Y\ar[r] & 0}\] 
in $\textup{mod}\,A,$
such that $X\in J^\perp$, $f_X$ is a $J$-approximation and $F(X)=Y$.
\end{lemma}

\begin{proof}
Let $Y\in \cmp\,B$, then there is an inclusion $\iota:Y\to P_B$ where $P_B$ is a projective $B$-module. 
Let $P_A$ be the corresponding projective $A$-module. Then sequence  \eqref{eq1} implies that $f_{P_A}$ is injective, and there exists a short exact sequence in $\text{mod}\,A$ appearing at the bottom of the following diagram,  

\begin{equation}\label{eq:2}
\xymatrix{0\ar[r] & J_{P_A} \ar[r]^{f_1}\ar@{=}[d] & X_1 \ar[r]\ar[d]^{\iota_1} & Y \ar[r] \ar[d]^{\iota}& 0\\
0\ar[r] & J_{P_A} \ar[r]^{f_{P_A}} & P_A \ar[r] & P_B \ar[r] & 0}
\end{equation}
where $J_{P_A}\in \add\,J$.  Note that $P_B=F(P_A)$.  The top row of the diagram is the short exact sequence obtained as the pullback along $\iota$.  Since $\iota$ is injective, it follows that $\iota_1$ is also injective.  Hence $X_1$ is a submodule of a projective $A$-module, so $X_1\in \cmp\,A$.  

We claim that $F(X_1)=Y$, hence we need to show that $f_1$ is an $\add\, J$ approximation of $X_1$.  Let $f: J'\to X_1$ where $J'\in \add\,J$. Since $f_{P_A}$ is an $\add\,J$ approximation of $P_A$ then $\iota_1f$ factors through $f_{P_A}$.  Thus, there exists some $u: J'\to J_{P_A}$ such that $f_{P_A} u = \iota_1f$.  By commutativity of the above diagram we have $\iota_1 f_1=f_{P_A}$, so we obtain $\iota_1f_1u=f_{P_A}u = \iota_1 f$.  Since $\iota_1$ is injective, we conclude that $f_1u=f$, so $f$ factors through $f_1$.  This shows the claim that $f_1$ is an $\add\,J$ approximation of $X_1$ and $F(X_1)=Y$.

Now, we want to show that $\Ext^1_A(J, X_1)=0$.  Applying the Snake Lemma to Diagram \eqref{eq:2}, we see that $\cok\,\iota\cong \cok\iota_1 \cong P_B/Y$.  Hence, there exists a short exact sequence as follows. 
\[0\to X_1 \to P_A \to P_B/Y\to 0\]
Applying $\Hom_A(J, - )$ to this sequence we obtain a long exact sequence
\[ \Hom_A(J, P_B/Y)\to \Ext^1_A(J, X_1)\to \Ext^1_A(J, P_A).\]
The first term is zero by Lemma~\ref{lem:J}(f) because $P_B/Y \in \text{mod}\,B$.  The last term is zero because $J\in \cmp\,A$, so we obtain the desired result that  $\Ext^1_A(J, X_1)=0$.  

If $\Ext^1_A(X_1, J)=0$, then by above $X_1\in J^\perp$.  Moreover, the sequence $0\to J_{P_A}\to X_1\to Y\to 0$ satisfies the statement of the lemma and we are done.  

Otherwise, if $\Ext^1_A(X_1, J)\not=0$ then we proceed as follows.  Let 
\begin{equation}\label{eq:3}
0\to J \to X_2\to X_1\to 0
\end{equation}
be a nonsplit short exact sequence in $\text{mod}\,A$.  Note that $X_2\in \cmp\,A$ since both $X_1, J\in \cmp\,A$ and the category $\cmp\,A$ is closed under extensions.  Consider the following diagram where the top row is a pullback of the above sequence along $f_{1}: J_{P_A}\to X_1$, the $\add\,J$ approximation of $X_1$.

\begin{equation}
\xymatrix{0\ar[r] & J \ar[r] \ar@{=}[d] & E \ar[r]\ar[d]^h & J_{P_A} \ar[r] \ar[d]^{f_{1}}& 0\\
0\ar[r] & J \ar[r] & X_2 \ar[r] & X_1 \ar[r] & 0}
\end{equation}
 
By Lemma~\ref{lem:J}(g) the module $J$ has no nontrivial selfextensions, thus $E \cong J\oplus J_{P_A}$.  By the Snake Lemma $\cok \,h \cong \cok f_{1} = F(X_1)$. 
 In particular, $\cok \,h$ is a $B$-module, and hence Lemma~\ref{lem largest quotient} implies that $\im \,f_{X_2} \subset \im \,h$.  Moreover, since $E\in \add\,J$ and $f_{X_2}$ is an $\add\,J$ approximation, then $\im \,f_{X_2} = \im \,h$.
   This implies that $F(X_2)\cong \cok\, h \cong \cok\, f_1\cong F(X_1)=Y$.   

Now, apply $\Hom_A(-,J)$ to \eqref{eq:3} to obtain 
\[\Hom_A(J,J)\to \Ext^1_A(X_1, J) \to \Ext^1_A(X_2, J)\to \Ext^1_A(J, J).\]
The last term is zero by Lemma~\ref{lem:J}(g), so the second map above is surjective. On the other hand, 
since the exact sequence (\ref{eq:3}) is non-split, the element $1_J \in \Hom_A(J,J)$ is send to a non-zero element under the first map in the above sequence. We conclude that $ \dim \Ext^1_A(X_1, J) >
\dim \Ext^1_A(X_2, J)$.  

Now, one can continue as before, replacing $X_1$ with $X_2$.  Eventually the process must terminate, as the dimension of the extension space between $X_j$ and $J$ is less then that between $X_{j-1}$ and $J$.  Hence there exists some $X_r=X\in J^\perp$ satisfying the statement of the lemma. 
\end{proof}

\begin{prop}\label{prop:fd}
The functor $F: J^\perp \to \cmp\,B$ is full and dense. 
\end{prop}

\begin{proof}
It follows directly from Lemma~\ref{lem:ses} that $F$ is dense.  To show that $F$ is full, consider a map $\bar f: F(M)\to F(N)$ where $M,N\in J^\perp$.  From the definition of $F$ and using
Lemma~\ref{lem:ses}, we see that
there exist two short exact sequences
\[\xymatrix@C20pt{0\ar[r]& J_M\ar[r]^{f_M} &M\ar[r]^-{\pi_M}&F(M)\ar[r]&0}\qquad \textup{and}\qquad
\xymatrix@C20pt{0\ar[r]& J_N\ar[r]^{f_N} &N\ar[r]^-{\pi_N}&F(N)\ar[r]&0} 
%0\to J_N\to N \xrightarrow{\pi_N} F(N)\to 0
\]
where $J_M, J_N\in \add\,J$.  Applying $\Hom_A(M, -)$ to the second sequence we obtain an exact sequence as follows. 
\[\Hom_A(M,N)\xrightarrow{{\pi_N}_*} \Hom_A(M, F(N))\to \Ext^1_A(M, J_N)\]
Note that the last term is zero because $M\in J^\perp$.
Therefore the composition $\bar f\circ \pi_M: M\to F(N)$  lifts to a map $f: M\to N$ such that the following diagram commutes. 
\[
\xymatrix{M\ar[r]^{\pi_M}\ar[d]^{f}& F(M)\ar[d]^{\bar f}\\
N\ar[r]^{\pi_N}& F(N)}
\]
By definition of the functor $F$, this implies that $F(f)=\bar f$.  This shows that $F$ is full. 
\end{proof}

Recall that $(\Fac J)$ denotes the ideal in $\cmp\, A$ of all morphisms factoring through $\Fac J \cap \cmp\, A$. 

\begin{thm}\label{thm:eq}
The functor $F: {J^\perp/(\Fac J)} \to \cmp\,B$ is an equivalence of categories. Moreover $F$ induces an equivalence $\bar F:\underline{J^\perp/(\Fac J)} \to \scmp\,B$. 
\end{thm}

\begin{proof} 
First, we show that $F(\Fac\, J)=0$.  Given $M\in \Fac \, J \cap \cmp\, A$,  there exists a surjection $J^d\to M$ for some $d\geq 0$.  In particular, the $\add\, J$-approximation of $M$ is surjective, which implies that $F(M)=0$.  
This shows that $F(\Fac\, J)=0$, and since $F$ maps projectives to projectives, 
 Proposition~\ref{prop:fd} implies that $\bar F$ is full and dense.
 
 It remains to show that $\bar F$ is faithful.  Suppose that $\bar F(f)=0$ in $\scmp\,B$ for some morphism $f: M\to N$ in $J^\perp$. Then for any morphism $f$, by definition of $F$ we have the following diagram 
\[\xymatrix{J_M \ar[r]^{f_M} \ar[d]& M \ar[r]^{\pi_M} \ar[d]^f& F(M) \ar[r] \ar[d]^{F(f)}& 0 \\
J_N \ar[r]^{f_N} & N \ar[r]^{\pi_N} & F(N) \ar[r] & 0}
\]
with exact rows.  

If $F(f)=0$ in $\cmp\,B$, then in particular $F(f)=0$ in $\text{mod}\,B$.  By commutativity of the above diagram $\pi_Nf =0$, so $f$ factors through the kernel of $\pi_N$.   Thus, there exists a map $u: M\to \ker\, \pi_N$ such that $f=i u$ where $i:\ker\,\pi_N\to N$ is the inclusion.   Since $\ker\,\pi_N=\im\, f_N$ is a quotient of $J_N$ and a submodule of $N$ then $\ker\,\pi_N\in \Fac J\cap \cmp\, A$.  In particular, $f$ is zero in $J^\perp/(\Fac J)$. This shows that $F$ is full, and from Proposition~\ref{prop:fd}we see that $F$ is an equivalence.  

If $F(f)=0$ in $\scmp\,B$ but it is a nonzero map in $\cmp\,B$, then $F(f)$ factors through a projective $B$-module.  We know from Lemma~\ref{lem:Fproj} that projective $A$-modules are mapped to projective $B$-modules under $F$. Moreover, by Proposition~\ref{prop:fd} the functor $F$ is full onto $\cmp\,B$, so $f$ factors through a projective $A$-module.  In particular, $f$ is zero in $\underline{J^\perp/(\Fac J)}$.
\end{proof}

As an application of Theorems \ref{prop:genrad} and \ref{thm:eq} we obtain several characterizations for when $A$ and $B$ have equivalent stable CM categories. 

\begin{corollary} \label{cor} Let $A=\kb Q/I$ be a 2-Calabi-Yau tilted algebra and $B = A/J$, where $J=Ae_iA$, with $i\in Q_0$. Further let $m_{ji}=\dim \Hom_A(P_A(i), P_A(j))=\dim \,e_j Ae_i$. Then the following are equivalent.
\begin{enumerate}
\item [(a)] $\bar F: \scmp\,A \to \scmp\,B$ is an equivalence.
\item [(b)] $J$ is projective in $\textup{mod}\,A$ and $m_{ii}=1$.
\item [(c)] $J\cong \oplus_{j\in Q_0} P_A(i)^{m_{ji}}$.
\item [(d)] For  all $j\in Q_0$, we have $P_A(i)^{m_{ji}}\subset P_A(j)$.
\item [(e)] The radial $\rad\, P(i)$ is generated by $\oplus_{j\not=i} \rad P(j)$.
\item [(f)] There exists a short exact sequence $0\to \rad\, P(i) \oplus Y\to P \to X\to 0$ where  $Y\in\cmp A$, $P$ is projective, and $X$ is not supported at $i$.
\end{enumerate}
\end{corollary}

\begin{proof}
By Theorem~\ref{thm:eq} we have $\bar F: \scmp\,A \to \scmp\,B$ is an equivalence if and only if $\underline{J^\perp/(\Fac J)}\cong \scmp\,A$.  This occurs if and only if $J=0$ and $\Fac J\cap \cmp\,A =0$ in $\scmp\,A$, which is equivalent to $J$ being projective in $\text{mod}\,A$ and $\Fac J\cap \cmp\,A$ containing only projective $A$-modules respectively. Thus to show that (a) and (b) are equivalent, it remains to show that $\Fac J\cap \cmp\,A $ consists only of projective $A$-modules if and only if $m_{ii}=\dim\, \Hom_A(P_A(i), P_A(i))=1$, provided $J$ is projective.  
%To show the forward implication, let $f\in \Hom_A(P_A(i), P_A(i))$.  Then $\im\,f \in \Fac J\cap \cmp\,A$ since it is a quotient of $P_A(i) \in \add\, J$ by Lemma~\ref{lem:J}(d) and a submodule of a projective module $P_A(i)$.  Therefore, if $f$ is nonzero and not an isomorphism then $\im\,f$ is not projective.   

Let $M\in \Fac J\cap \cmp\,A $ and we claim that $M$ can be realized as an image of some map $g$ in $\add\, J$.
Since $M\in \Fac J$ there exists a surjection $\pi: J_M \to M$ where $J_M \in \add\, J$, and since $M\in \cmp\, A$ there exists an inclusion $u: M \to P$  where $P$ a projective $A$-module.  Then we have the following diagram with exact rows.  Here $f_P$ is the $\add \, J$ approximation of $P$, which is injective by definition of $J$, see the short exact sequence \eqref{eq1}.

\[\xymatrix{
&J_M \ar[rr] \ar@{->>}[dr]^{\pi} \ar@{-->}[dd]^{g}&  & P \ar[rr]^{v} \ar@{=}[dd]&& \cok \, u \ar[r] \ar@{-->>}[dd]^{h}& 0 \\
&&M \ar@{^{(}->}[ur]^{u} \ar@{..>}[dl]^q\\
0\ar[r]&J_P \ar@{^{(}->}[rr]^{f_{P}}&& P \ar[rr]^{\pi_{P}} && F(P) \ar[r] & 0 
}
\]

Since $u\pi$ is a map from $\add\, J$ to $P$ then it factors through $f_P$ and there exists a map $g: J_M\to J_P$ that makes the square on the left commute.  Then $\pi_P u\pi = \pi_P f_Pg=0$, and by the universal property of the cokernel of $u\pi$, there exits a map $h: \cok \, u \to F(P)$ that makes the square on the right commute.  Now by the universal property of the kernel of $\pi_P$ there exists a map $q: M\to J_P$ such that $u=f_P q$.  Note that $q$ must be injective.   Since $f_P q \pi = u\pi = f_P g$ and $f_P$ is injective we conclude that $q \pi = g$.  In particular, $M \cong \im\, g$, which shows the desired claim.  

Now, we show that given $J$ is projective, then $\Fac J\cap \cmp\,A $ consists only of projective $A$-modules if and only if $m_{ii}=1$.  First note that since $J$ is projective then Lemma~\ref{lem:J}(e) implies that $J$ is a direct sum of copies of $P_A(i)$. Moreover, the image of any map $g: P_A(i)\to P_A(i)$ is a module $M\in  \Fac J\cap \cmp\,A$.   Therefore, if $m_{ii}>1$ then there exists a map $g: P_A(i)\to P_A(i)$ that is nonzero, not an isomorphism, and whose image $M$ is not a projective module. Conversely, since any $M\in  \Fac J\cap \cmp\,A$ can be realized as an image of some map $g$ in $\add\, J=\add \, P_A(i)$, then $m_{ii}=1$ implies that $M\in \add\, P_A(i)$.  In particular, $\Fac J\cap \cmp\,A $ consists only of projective $A$-modules.
This completes the proof that (a) and (b) are equivalent. 

Next, we show that (b) implies (c).  
Since $J$ is projective then, by Lemma~\ref{lem:J}(e), it is a direct sum of copies of $P_A(i)\cong e_iA$.  By Lemma~\ref{lem:J}(b) we have $J \cong \sum_w wA$ where $w$ runs over all equivalence classes of paths in $A$ ending at vertex $i$.  For a fixed path $w$, the module $wA$ is a submodule of $J$ and also $wA$ is a quotient of $P_A(i)$.  Then $wA$ is an image of some map $J \to P_A(i)$.  Since $m_{ii}=1$, we conclude that $wA\cong P_A(i)$.  Therefore, we have that 
\[J \cong  \sum_w wA \cong \sum_{j\in Q_0} P_A(i)^{m_{ji}}.\]
It remains to show that this sum is a direct sum.  Thus, assume that there exist paths $w, w' \in e_h A e_i$ and paths $a, a' \in e_i A e_j$   for some vertices $h, j \in Q_0$ such that $wa=w'a'$ in the algebra $A$. Without loss of generality we may assume that $w, w'$ only visit vertex $i$ at the endpoints, otherwise we can replace them with starting segments of $w, w'$ that only visit vertex $i$ at the endpoints.  Then $P_A(i)\oplus P_A(i) \xrightarrow{(w, w')} P_A(h)$ is part of the injective map $0\to e_h J \to e_h A$.  But the element $(a, a')^T$ would be in the kernel of this map, which is a contradiction.  This completes the proof that (b) implies (c).

Now suppose (c) holds. Then $J \cong Ae_i A \cong \oplus_{j\in Q_0} P(i)^{m_{ji}}$ as right $A$-modules. By Lemma~\ref{lem:J}(b) we also have $J \cong \sum_w wA$ where $w$ runs over all equivalence classes of paths in $A$ ending at vertex $i$.  
Since the number of direct summands in the first expression for $J$ is equal to the number of paths $w$ in the second expression, then the second sum is also a direct sum.  In particular, $J\cong \oplus _w wA$.
Multiplying the sequence \eqref{eq1} with $e_i$ on the left yields a short exact sequence in $\text{mod}\,A$ 
\[0\to e_i J \to e_i A \to e_i B \to 0\]
where the last term is zero, since $B$ is not supported at $i$.  Then 
\[e_iA\cong e_i J  = e_i \big( \oplus_w wA) = \oplus_w e_i wA\] 
where $w$ runs over all equivalence classes of paths in $A$ ending at $i$. Thus there is only one path $w=e_i$.  
This shows the equivalence of (b) and (c).

 %and by Lemma~\ref{lem:J}(b), this is equivalent to  $w A\cong P_A(i)$ for all paths $w$ in $A$ that are ending in $i$. By definition, $m_{ji}$ is the number of paths from $j$ to $i$. Thus, using Lemma~\ref{lem:J}(c), $J$ is projective if and only if $e_jJ\cong P_A(i)^{m_{ji}}$, for all $j$. This shows the equivalence of (b) and (c). 
 
Condition (c) implies  (d) because $e_jJ=e_jAe_iA\subset e_j A=P_A(j)$ for all $j\in Q_0$.  On the other hand, if $P_A(i)^{m_{ji}}\subset P_A(j)$ then whenever $u$ is a path from $j$ to $i$ and $v$ is a path starting at $i$ then the composition $uv$ is nonzero in $A$. Thus $e_jJ=e_jAe_iA\cong P_A(i)^{m_{ji}}$, and hence (d) implies (c).

The conditions (e) and (f) are equivalent by Lemma~\ref{genrad}. Now, we show that (e) implies (b).  Suppose that the radical $\rad\, P(i)$ is generated by the other radicals $\oplus_{j\not=i} \rad P(j)$.  Then by Lemma~\ref{genrad} there exists a short exact sequence 
\begin{equation}\label{eq:21}
0\to \rad\,P(i) \oplus Y\xrightarrow{u} P \to X\to 0
\end{equation}
where  $Y\in\cmp A$, $P$ is projective, and $X$ is not supported at vertex $i$.  Applying $\Hom_A(J, -)$ to this sequence, where $J= Ae_iA$, we obtain the following long exact sequence.
\[\dots \to \Hom_A(J, X)\to \Ext^1_A(J, \rad\,P(i) \oplus Y)\to \Ext^1_A(J, P)\to \dots\] 
Note that since $X$ is not supported at $i$ then Lemma~\ref{lem:J}(e) implies that $\Hom_A(J, X)=0$. Moreover, by Lemma~\ref{lem:J}(a) we have $J\in \cmp A$, and thus $\Ext^1_A(J, P)=0$.  Thus, we conclude that $\Ext^1_A(J, \rad\,P(i))=0$.  Since $J$ is a syzygy whose top consists of simple modules at vertex $i$, it follows that the projective cover $P(J)$ of $J$ consists of projective modules $P(i)$.  However,  \cite[Lemma 5.2]{SS4} implies that every non-projective syzygy $J$ has an extension with the radical of its projective cover. Thus $\Ext^1_A(J, \rad\,P(i))$ being zero implies $J$ is projective and $J^\perp = \cmp\, A$.  It remains to show that $m_{ii}=1$.   Consider the following commutative diagram with exact rows, where the map $u$ is the inclusion appearing in sequence \eqref{eq:21}.  The square on the right is obtained by applying the functor $F$ to $u$.  The square on the left is obtained by taking kernels of $\pi$ and $\pi_P$ respectively, and note that since $P$ is a projective $A$-module, then $F(P)$ is a projective $B$-module by Lemma~\ref{lem:Fproj}, so the kernel of the projection $\pi_P$, denoted by $J_P$ is a summand of $J$. 
\[\xymatrix{
0\ar[r] & \ker\,\pi \ar[r] \ar@{^{(}..>}[d]^v& \rad\, P(i) \oplus Y \ar[r]^{\pi} \ar@{^{(}->}[d]^u& F(\rad\, P(i) \oplus Y) \ar[r] \ar[d]^{F(u)}& 0 \\
0\ar[r] & J_P\ar[r] & P \ar[r]^{\pi_P} & F(P) \ar[r] & 0 
}
\]
Now applying the Snake Lemma we obtain a long exact sequence, where we note that $\cok \, u \cong X$.  
\[0\to \ker\, F(u) \to \cok\, v \to X \to \cok\, F(u)\to 0 \]
Recall that $X$ and $\ker\, F(u)$ are not supported at vertex $i$, while $\cok\, v$ is a quotient of $J_P$, so its top is a direct sum of copies of $S(i)$.  This implies that $\cok\, v = 0$, and therefore $\ker\,\pi \cong J_P$.   In particular, the $\add\, J$-approximation of $\rad\, P(i)$ is a summand of $J_P$ and is an inclusion.  Since $J$ is projective, we obtain an inclusion $P_A(i)^a \to \rad\, P(i)$ for some $a\geq 0$.  Moreover, $\rad\, P(i)$ is a proper submodule of $P_A(i)$, so we obtain a proper inclusion $P_A(i)^a \to P_A(i)$.  This implies that $a=0$.  Since,  $P_A(i)^a = 0$ is an $\add\, J$-approximation of $\rad\,P(i)$ it follows that $\rad\,P(i)$ is not supported at $i$.  Therefore, we conclude that $m_{ii}=1$.    This completes the proof that (e) implies (b).  

Lastly, it remains to show that (a) implies (e). Suppose that $\bar F: \scmp\,A \to \scmp\,B$ is an equivalence. Since (a) implies (b) then $J$ is a direct sum of copies of $P_A(i)$ and $\Fac J \cap \cmp\, A= \add\, J$.  In particular, Theorem~\ref{thm:eq} implies that $F: \cmp\, A/ \add\, P_A(i) \to \cmp\, B$ is an equivalence.  Furthermore, by Theorem~\ref{prop:genrad}, the  radicals of $B$ generate $\scmp\,B$, which under the above equivalence means that the radicals $\rad P_A(i)$, with $i\ne j$, generate $\scmp\,A$, provided that a short exact sequence in  $\cmp\, B$ lifts to a short exact sequence in $\cmp\, A$ via $F$.  In particular, this would imply that $\oplus_{j\not=i} \rad P(j)$ generate $\rad\,P(i)$, which shows (e).  Therefore, it remains to show the claim that for a short exact sequence 
\begin{equation}\label{eq:22}
0\to X' \xrightarrow{u'} Y' \xrightarrow{v'} Z'\to 0
\end{equation}
in $\cmp\, B$, there exists a short exact sequence in $\cmp\, A$ that maps to \eqref{eq:22} under $F$, assuming condition (b) holds.  Since $F$ is an equivalence, there exist $v:Y\to Z \in \cmp\, A$ such that $F(v)=v'$.   Then we obtain the following commutative diagram with exact rows.  Note that $f_Y, f_Z$ are injective because $\ker\pi_Y, \ker\,\pi_Z \in \Fac J \cap \cmp\, A= \add\, J$.
\[\xymatrix{
0\ar[r] & J_Y \ar[r]^{f_Y} \ar[d]^h& Y\ar[r]^{\pi_Y} \ar[d]^v& Y' \ar[r] \ar[d]^{v'}& 0 \\
0\ar[r] & J_Z \ar[r]^{f_Z} & Z\ar[r]^{\pi_Z} & Z' \ar[r] & 0 \\
}\]
Since, $J$ is a direct sum of copies of $P_A(i)$ and $\End_A(P_A(i))=\kb$, it follows that $\cok \, h$ is a summand of $J_Z$ and $\ker\, h$ is a summand of $J_Y$.  To emphasize this, denote $J'_Z = \cok \, h$ and  $J'_Y=\ker\, h$.   Since $v'$ is surjective by assumption, the Snake Lemma implies that $p': J_Z' \to \cok \, v$ is surjective.  
Let $p: Z \to \cok \, v$ denote the surjection, and since $J_Z'$ is a summand of $J_Z$ let $i: J_Z'\to J$ denote the corresponding inclusion. Note that $p'=pf_Z i$.  Now, define a map $v'': Y\oplus J_Z' \to Z$ by setting $v''= \begin{bsmallmatrix}v && f_Z i \end{bsmallmatrix}$.   
We claim that $v''$ is surjective.  Given $z\in Z$, consider $p(z)\in \cok \, v$.  Since $p'$ is surjective, there exists $z'\in J_Z'$ such that $p'(z')=p(z)$.  Then $z-f_Z i(z')\in \ker\, p = \im\, v$, and there exists $y\in Y$ such that $v(y)=z-f_Z i(z')$.  Then $z=v(y)+f_Zi(z')$, which shows that $v''$ is surjective.

Let $X$ denote the kernel of $v''$.  Now, consider the following commutative diagram with exact rows and columns derived from the diagram above by replacing $v$ with $v''$.  Note that the square on the bottom right commutes, because there is no nonzero map from $J_Z'$ to  $Z'$ since $Z'\in \cmp B$ is not supported at vertex $i$ while $J_Z'\in \add \, J$ has top only consisting of summands of $S(i)$.   Moreover, the square on the bottom left commutes by definition of $v''$ and the equality $f_Zh=vf_Y$, which holds by the diagram above.
\[\xymatrix{
&0\ar[d]&0\ar[d]&0\ar[d]\\
0\ar[r]&J_Y'\ar[d]\ar[r] & X \ar[r]\ar[d]^{u''} & X'\ar[r]\ar[d]^{u'} & 0 \\
0\ar[r] & J_Y\oplus J_Z' \ar[r]^{\begin{bsmallmatrix}f_Y&0\\0&1\end{bsmallmatrix}} \ar[d]^{\begin{bsmallmatrix}h &i \end{bsmallmatrix}}& Y\oplus J_Z' \ar[r]^{[\pi_Y \,\, 0]} \ar[d]^{v''}& Y' \ar[r] \ar[d]^{v'}& 0 \\
0\ar[r] & J_Z \ar[r]^{f_Z} \ar[d]& Z\ar[r]^{\pi_Z} \ar[d]& Z' \ar[r] \ar[d]& 0 \\
&0&0&0
}\]
Observe that the short exact sequence in the middle column is in $\cmp\, A$, since $Y, Z, J_Z' \in \cmp\, A$ and $X$ is a submodule of $Y\oplus J_Z'$.  Furthermore, the sequence in the third column is precisely the short exact sequence given in equation \eqref{eq:22}, and by Lemma~\ref{lem largest quotient} it is obtained from the one in the middle by applying the functor $F$, since the short exact sequence in the first column consists of summands of $J$.   This completes the proof of the claim. 
\end{proof}

\begin{remark}
The statement that (b) implies (a) is a special case of the main theorem of Chen's paper \cite{Chen}, because the condition (b) implies that $J$ is a homological ideal of finite projective dimension.
Also, the generalization of Lu's algorithm \cite{Lu} appearing in \cite{SS4} which involves adding/removing vertices while preserving the syzygy category is  a special case of the corollary.
\end{remark}

\begin{example}\label{ex1}
 
 We illustrate the theorem in an example. Let $A$ be the  algebra given by the quiver on the left below with relations $\za\zb=\zg\zd, \ \ze\za=\zs\tau, \ \ze\zg=\zd\ze=\zb\ze=0, \ \zb\zs=\tau\zb=0 $,
and let $J=Ae_5 A$. Then $B$ is the algebra given by the quiver on the right below with relations $\za\zb=\zg\zd,\ \ze\za= \ze\zg=\zd\ze=\zb\ze=0$. 
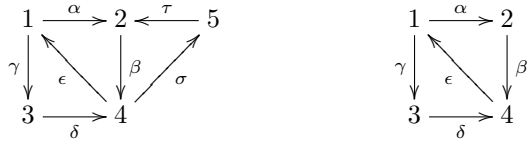
\[\xymatrix{
1\ar[r]^\za\ar[d]_\zg&2\ar[d]^\zb&5\ar[l]_\tau\\
3\ar[r]_\zd&4\ar[ul]^\ze\ar[ur]_\zs
}
\qquad\qquad\qquad
\xymatrix{
1\ar[r]^\za\ar[d]_\zg&2\ar[d]^\zb\\
3\ar[r]_\zd&4\ar[ul]^\ze
}
\]
%%%%%%%%%%%%%%
%%%%%%%%%%%%%%
%%%%%%%%%%%%%%
%%%%%%%%%%%%%%
 The right $A$-module $J$ is equal to \[J=e_5A\oplus \zs A\oplus \zd\zs A = P(5)\oplus P(5)\oplus S(5)=\begin{smallmatrix}
5\\2
\end{smallmatrix}\oplus \begin{smallmatrix}
5\\2
\end{smallmatrix}\oplus 5.\]
The category $\cmp A$ is shown below. The subcategory $J^\perp$ is given by the modules in red and the  category $\underline{J^\perp/(\Fac J)}$ is given by the underlined modules.
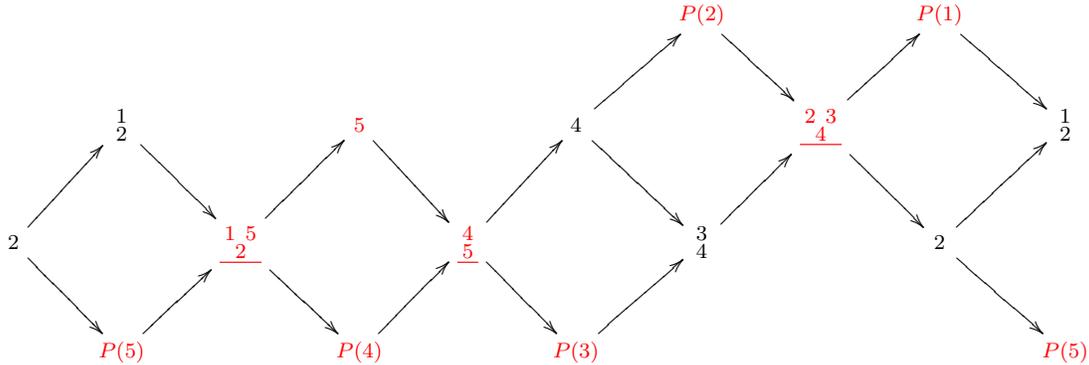
\[\large\xymatrix@C20pt{
&&&&&&
{\begin{smallmatrix}
\color{red} P(2)
\end{smallmatrix}}
\ar[rd] &&{\begin{smallmatrix}
\color{red} P(1)
\end{smallmatrix}}
\ar[rd] 
\\
%Row 2
&{\begin{smallmatrix}
1\\2
\end{smallmatrix}}
\ar[rd] 
&&
{\begin{smallmatrix}
\color{red} 5
\end{smallmatrix}}
\ar[rd] 
&&
{\begin{smallmatrix}
4
\end{smallmatrix}}
\ar[rd] \ar[ru]
&&
\underline{\color{red}\begin{smallmatrix}
 2\ 3\\4
\end{smallmatrix}}
\ar[rd] \ar[ru]
&&
{\begin{smallmatrix}
1\\2
\end{smallmatrix}}
&&
\\
% Row 3
{\begin{smallmatrix}
2
\end{smallmatrix}}
\ar[rd] \ar[ru]
&&
\underline{\color{red} \begin{smallmatrix}
1\ 5\\2
\end{smallmatrix}}
\ar[rd] \ar[ru]
&&
\underline{\color{red} \begin{smallmatrix}
4\\ 5
\end{smallmatrix}}
\ar[rd] \ar[ru]
&&
{\begin{smallmatrix}
3\\4
\end{smallmatrix}}
 \ar[ru]
&&
{\begin{smallmatrix}
2
\end{smallmatrix}}
\ar[rd] \ar[ru]
&&
\\
%Row 4
&
{\color{red} \begin{smallmatrix}
P(5)
\end{smallmatrix}}\ar[ru]
&&
{\color{red} \begin{smallmatrix}
P(4)
\end{smallmatrix}}\ar[ru]
&&
{\color{red} \begin{smallmatrix}
P(3)
\end{smallmatrix}}\ar[ru]
&& &&
{\begin{smallmatrix}
\color{red} P(5)
\end{smallmatrix}}
}
\]
%%% B
The category $\cmp B$ is shown below. The modules in $\scmp B$ are underlined. To compute the image of the functor $F\colon \underline{J^\perp/(\Fac J)}\to \scmp B$ on the module $X=\begin{smallmatrix}
1\ 5\\2
\end{smallmatrix}$, we see that the $\add \,J$ approximation is given by the inclusion $f_X\colon P(5)\to \begin{smallmatrix}
1\ 5\\2
\end{smallmatrix}$ and thus $F(X)=\cok(f_X) =1$.
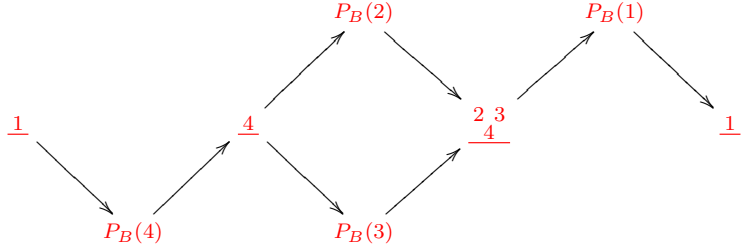
\[\large\xymatrix@C20pt{
%Row 2
& 
&&
&&{\begin{smallmatrix} 
\color{red} P_B(2)
\end{smallmatrix}}\ar[rd]
&&
{\color{red}\begin{smallmatrix}
P_B(1)
\end{smallmatrix}}
\ar[rd] 
\\
% Row 3
&&\underline{\color{red} \begin{smallmatrix}
1
\end{smallmatrix}}
\ar[rd] 
&&
\underline{\color{red} \begin{smallmatrix}
4
\end{smallmatrix}}
\ar[rd] \ar[ru]
&&
\underline{\color{red}\begin{smallmatrix}
2\ 3\\4
\end{smallmatrix}}
 \ar[ru]
&&\underline{\color{red}\begin{smallmatrix}
1
\end{smallmatrix}}
&&
\\
%Row 4
&
&&
{\color{red} \begin{smallmatrix}
P_B(4)
\end{smallmatrix}}\ar[ru]
&&
{\color{red} \begin{smallmatrix}
P_B(3)
\end{smallmatrix}}\ar[ru]
&& &&
}
\]
\end{example}

The following example illustrates some of the subtleties involved in Theorem~\ref{thm:eq} and Corollary~\ref{cor}.

\begin{example}
Let $A$ be an algebra given by the following quiver with relation $\alpha^2=0$.  Note that $A$ is a Jacobian algebra given by the quiver with potential $W=\alpha^3$ where we assume that the characteristic of the field $\kb$ is not 3. 
%\vspace{.1cm}
\[\xymatrix{ 2 \ar[r]^\beta & 1 \ar@(ul,ur)^{\alpha}}\]
Let $J = Ae_1A$ and $B = A/J$.   Then as a right $A$-module we have $J \cong  e_1A\oplus \beta A = P_A(1)\oplus P_A(1)$, so $J$ is projective.  Also $J$ can be written as a sum  $J = e_1 A + \beta A + \alpha A+\beta\alpha A$, as in Lemma~\ref{lem:J}(b), but this is not a direct sum decomposition since    $\alpha A\subset e_1 A$ and $\beta\alpha A \subset \beta A$.  The category $\cmp\, A$ consists of three indecomposable modules $P_A(2)={\begin{smallmatrix}2\\1\\1\end{smallmatrix}}, P_A(1)={\begin{smallmatrix}1\\1\end{smallmatrix}}$ and the simple module $S(1)$, while $\Fac J \cap \cmp\, A = \add P_A(1)\oplus S(1)$.   Then $J^\perp/ (\Fac J)$ consists of $P_A(2)$.  On the other hand, $\cmp\, B$ consists only of the simple module $S(2)=F(P_A(2))$.  This shows the equivalence $F$ from Theorem~\ref{thm:eq}.

Note also that $\scmp\, A$ and $\scmp\, B$ are not equivalent as the first category consists of a single indecomposable module $S(1)$  while the second category is zero.  In particular, condition (a) of Corollary~\ref{cor} is not satisfied, and equivalently condition (b) is not satisfied because $m_{11}=2$.  
\end{example}

\section{Applications in finite CM-type}\label{sect 4}

In this section we apply Corollary~\ref{cor} to dimer tree algebras and their skew group algebras, which are certain classes of 2-Calabi-Yau tilted algebras of finite CM-type.   In particular, we obtain explicit characterizations in terms of quivers when a reduction at an idempotent preserves the $\scmp$ category of these algebras, see Theorem~\ref{thm A reduction} and Theorem~\ref{thm D}.

\subsection{Reduction of dimer tree algebras} 
 Dimer tree algebras are a special class of 2-Calabi-Yau tilted algebras that are of CM-type $\mathbb{A}$. These algebras were introduced in \cite{SS3} as Jacobian algebras of a quiver $Q$ with potential $W$, where $Q$ is a finite connected quiver without loops and 2-cycles such that every arrow of $Q$ lies in at least one chordless cycle and, furthermore, the dual graph of $Q$ is a tree. It follows from the definition that $Q$ is a planar quiver, and we define the potential $W$ to be the signed sum of the chordless cycles in $Q$, where the sign of a clockwise cycle is +  and the sign of a counterclockwise cycle is -. An arrow in $Q$ is a \emph{boundary} arrow if it lies in exactly one chordless cycle, and it is an \emph{interior} arrow otherwise; in this case it lies in exactly two chordless cycles. 
 
For every boundary arrow $\zb$ of $Q$, the \emph{maximal zigzag path} (or cycle path) starting with $\zb$ is the unique path $\zb_1\zb_2\ldots\zb_{\ell(\zb)}$ such that 
 
\begin{itemize}
 \item $\zb_1=\zb$ and $\zb_{\ell(\zb)}$ are boundary arrows, and $\zb_2,\ldots,\zb_{\ell(\zb)-1}$ are interior arrows;
\item every subpath of length two $\zb_i\zb_{i+1}$ is a subpath of a chordless cycle and no subpath of length three is a subpath of a chordless cycle.
\end{itemize}
The  
  \emph{weight} $\wt(\zb)$ of a boundary arrow $\zb$ is equal to 1, if the length $\ell(\zb)$ of the zigzag path starting with $\zb$ is odd, and equal to 2, if $\ell(\zb)$ is even. Dually, the coweight $\overline{\wt}(\zb)$ is equal to 1 or 2 depending on the parity of the length of the unique maximal zigzag path ending with $\zb$. 
 
 In Example~\ref{ex1}, the zigzag path starting with $\tau$ is $\tau \zb\ze\zg$ and the zigzag path ending with $\tau$ is $\zs\tau$. Thus the weight $\wt(\tau)$ is 2 and the coweight $\overline{\wt}(\tau)=2$. 

\smallskip

As a  first application of our main theorem, we have the following strengthening of a result from \cite{SS4}. 

\begin{thm}
 \label{thm A reduction}
 Let $A$ be a dimer tree algebra, let $i$ be a vertex in its quiver and denote by $J=Ae_iA$ the corresponding 2-sided ideal.  Then the following are equivalent. 
 
 (a) $A$ and $A/J$ have the same CM-type.
 
 (b) There exists a 3-cycle in $Q$ of the form $\xymatrix{h\ar[r]^\zb&i\ar[r]^\zg&j\ar@/^10pt/[ll]^\zd}$ with $\zb$ and $\zg$ boundary arrows and $\zd$ an interior arrow such that $\overline {\wt}(\zb)=1$ and $\wt(\zg)=1$. 
\end{thm}
\begin{proof}
 The implication (b)$\Rightarrow$(a) was shown in \cite[Proposition 4.10]{SS4}.  
 
 Conversely, assume (a) holds. Let $\zb\colon h\to i$ be an arrow ending at $i$. If $\zb$ is not a boundary arrow then $\zb$ is contained in two chordless cycles $u\zb$ and $v\zb$ of $A$, and therefore the relation $\partial_\zb W= u-v$ given by the partial derivative of the potential in direction $\zb$ is a commutativity relation. In particular, $u=v$ are nonzero paths from $i$ to $h$ in $A$\footnote{We use here the fact that, since $A$ is a dimer tree algebra, the closure of the Jacobian ideal is equal to the Jacobian ideal itself. Thus the fact that $u$ does not lie in the ideal generated by the relations implies that $u$ is nonzero in $A$.}, and hence the vertex $h$ lies in the support of $P(i)$. 
 
 On the other hand, $\zb$ is a nonzero path from $h$ to $i$, thus the vertex $i$ lies in the support of $P(h)$. By Corollary~\ref{cor}(d), this implies that $P(i)\subset \rad P(h)$. However, from the above this would imply that the radical of $P(h)$ is supported at $h$, and hence there is a nonzero, non-constant path from $h$ to $h$. This is impossible since $A$ is Schurian by \cite[Corollary 3.32]{SS3}.  Thus every arrow that ends in $i$ is a boundary arrow.

 Now let $\zg\colon i\to j$ be an arrow starting at $i$.  Then $j$ lies in the support of $P(i)$. If $\zg$ is an interior arrow, then by the same argument as above, we have $i$ in the support of $P(j)$. Corollary~\ref{cor} implies $P(i)\subset \rad P(j)$, and hence also $j$ lies in the support of $\rad P(j)$. Again, this is impossible since $A$ is Schurian. 
 
 Therefore, there is a chordless cycle
 \[\xymatrix{i=k_1\ar[r]^{\zg=\zd_1} &j=k_2\ar[r]^{\zd_2}
 &k_3\ar[r]&\cdots \ar[r]&k_{\ell-1}\ar[r]^{\zd_{\ell-1}}
 &h=k_{\ell}\ar[r]^{\zb=\zd_{\ell}}&i=k_1}\]
 of length $\ell\ge 3$, with $\zb,\zg$ boundary arrows. Then $P(i)$ contains $k_1,k_2,\ldots,k_{\ell-1}$ in its support, and $P(k_3)$ contains $k_3,k_4,\ldots, h,i$.  If $\ell\ge 4$, then  $P(i)$ contains $k_3$ and $P(k_3)$ contains $i$. Again using Corollary~\ref{cor}, we see $P(i)\subset P(k_3)$. Thus if $\ell\ge 4$, then there is a cyclic path from $k_3 $ to $k_3$ which is impossible, since $A$ is Schurian. 
 
 Thus $\ell=3$ and we have a cycle of the form $\xymatrix{h\ar[r]^\zb&i\ar[r]^\zg&j\ar@/^10pt/[ll]^\zd}$ with $\zb$ and $\zg$ boundary arrows. If $\zd$ were also a boundary arrow, then this 3-cycle is the whole quiver of $A$ and its CM-type would be $\mathbb{A}_1$. However, the algebra $A/J$ is then without oriented cycles and of CM-type $\mathbb{A}_0$, which is a contradiction to part (a). Thus $\zd$ is an interior arrow. 
 This shows the existence and uniqueness of the 3-cycle in part (b). Furthermore, the conditions on the weight and coweight in (b) hold by \cite[Rem 4.11]{SS4}.  
\end{proof}

\begin{example}
Let $A$ be the dimer tree algebras with the following quiver.
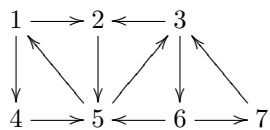
\[\xymatrix@R25pt@C20pt{1\ar[r]&2\ar[d]&3\ar[l]\\
4\ar@{<-}[u]&5\ar@{<-}[l]\ar[lu]\ar[ru]\ar@{<-}[r]&6\ar@{<-}[u]\ar[r]&7\ar[lu]}\]
%\qquad 
%\xymatrix@R15pt@C12pt{1\ar[r]&2\ar[d]&3\ar[l]\\
%4\ar@{<-}[u]&5\ar@{<-}[l]\ar[lu]\ar[ru]\ar@{<-}[r]&6\ar@{<-}[u]}
%\qquad 
%\xymatrix@R15pt@C12pt{1\ar[r]&2\ar[d]&3\ar[l]\\
%&5\ar[lu]\ar[ru]\ar@{<-}[r]&6\ar@{<-}[u]\ar[r]&7\ar[lu]}\]
Let $B=A/Ae_7A$ be the algebra obtained by the reduction at vertex 7 and $C=A/Ae_4A$ by reduction at vertex 4. Then $A$ and $B $ have the same CM-type, because the vertex $7$ satisfies condition (b). Indeed the zigzag paths are 
$7\to 3\to 6 \to 5$ and $1\to 2\to5\to3\to6\to7$, and both have an odd number of arrows, so the (co)weight conditions in (b) are satisfied.

In contrast, the algebras $A$ and $C $ do not have the same CM-type, because the vertex $4$ does not satisfy condition (b). Here the zigzag path $3\to 2\to 5\to1\to 4$ ending at 4 has an even number of arrows, hence the coweight condition in (b) is not satisfied. 

Since 4 and 7 are the only vertices that lie in a boundary 3-cycle, no other reduction of $A$ will preserve the CM-type.
\end{example}
\subsection{Minimal dimer tree algebras have minimal skew group algebras}
 We recall a construction from \cite[Section 4]{SS5} that associates a family of 2-Calabi-Yau tilted algebra of CM-type $\mathbb{D} $ to every dimer tree algebra. 

Let $A^0=\textup{Jac}(Q^0,W^0)$ be a dimer tree algebra given as the Jacobian algebra of a quiver $Q^0$ with potential $W^0$.  We assume throughout that $Q^0$ has more than 3 vertices. Let $\za\colon 1\to 2$ be a boundary arrow in $Q^0$. 

Define $Q$ to be the quiver obtained by taking two copies $Q^0,(Q^{0})'$ of $Q^0$ and gluing them along the arrow $\za$. Thus every vertex $i\ne 1,2$ of $Q^0$ gives rise to two vertices $i,i'$ in $Q$, and every arrow $\zb\ne \za$ of $Q^0$ gives rise to two arrows $\zb,\zb'$ in $Q$. 

Let $G=\{1,\zs\}$ be the group with two elements, and define a $G$-action on $Q$ by
\[\zs\cdot i=i', \ \zs\cdot i' =i,\  \zs\cdot \zb=\zb', \  \zs\cdot \zb'=\zb,\]
for vertices $i,i'\ne 1,2$ and arrows $\zb,\zb'\ne \za$, and
\[\zs\cdot 1 =1,\ \zs\cdot 2 =2,\ \zs\cdot \za =\za.\]
 Let $W$ be the potential on $Q$ defined by $W=W^0-\zs W^0$, and define $A=\textup{Jac}(Q,W)$. The algebra $A$ is called the \emph{fibered product} of $A^0$ with itself along the  arrow $\za\colon 1\to 2$.  Because of our assumption that $Q^0$ has more than three vertices, $Q$ is \emph{not} following quiver.
 
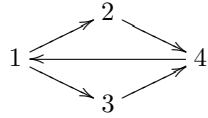
\begin{equation}\label{not Q}
{\xymatrix@R5pt{&2\ar[rd]\\1\ar[ru]\ar[rd]&&4\ar[ll]\\&3\ar[ru]}}
\end{equation}
We need to exclude this quiver because it is too small for the construction.
 
 The $G$-action on $Q$ induces an \emph{admissible} $G$-action on $A$, meaning that $\zs$ maps vertices to vertices and arrows to arrows, and $\zs$ fixes at least one vertex. 
 
 It is shown in \cite[Propositions 4.9, 4.10]{SS5} that $A$ is also a dimer tree algebra and  that every dimer tree algebra with a non-trivial admissible $G$-action arises this way.

Let $B=AG$ be the skew group algebra of $A$ with respect to the $G$-action as defined in \cite[Section 2.6]{SS5}. The quiver $Q_B$ of $B$ has one vertex $i$ for every $G$-orbit $\{i,\zs i\}$ of cardinality two in $Q$ and four vertices $1^+,1^-,2^+,2^-$  for the two vertices $1,2$ of $Q$ whose $G$-orbits are trivial. 
Moreover, $Q_B$ has one arrow $\zb\colon i\to j$ for every $G$-orbit $\{\zb,\zs\zb\}$ in $Q$  of cardinality two such that the arrows $\zb,\zs\zb$ are not incident to the vertices 1 and 2, and two arrows $\za^\pm\colon 1^\pm\to 2^\pm$ for the arrow $\za\colon 1\to 2 $ of $Q$, as well as two arrows $\zb^\pm$ for every arrow $\zb\ne \za$ of $Q$ that is incident  to the vertex 1 or 2. 

\begin{example} The quiver on the left is the quiver $Q$ of a dimer tree algebra $A$ with $G$-action indicated by the primes. The quiver on the right  is the quiver $Q_B$ of the (basic) skew group algebra $B$.
 \[\xymatrix{
7\ar[r]&4\ar[d]&3\ar[l]\ar[r]&5\ar[ld]&6\ar[l]\\
&1\ar[lu]\ar[r]^\za\ar[ld]&2\ar[d]\ar[u]\ar[rru]\ar[rrd]\\
7'\ar[r]&4'\ar[u]&3'\ar[l]\ar[r]&5'\ar[lu]&6'\ar[l]\\
}
\qquad\qquad\qquad
\xymatrix{
7\ar[r]&4\ar[d]\ar@/_10pt/[dd]&3\ar[l]\ar[r]&5\ar@/_
0pt/[ddl]\ar[ld]&6\ar[l]\\
&1^+\ar[lu]\ar[r]^{\za^+}&2^+\ar[u]\ar[rru]\\
&1^-\ar[luu]\ar[r]^{\za^-}&2^-\ar@/^10pt/[uu]\ar[rruu]\\
}\]

\end{example}

In \cite[Definition 3.18]{SS3}, we defined the \emph{total weight} of a dimer tree algebra as the sum of the weights of all boundary arrows. We showed that the CM-type of a dimer tree algebra of total weight $2N$ is  $\mathbb{A}_{N-2}$.
One of the main results of \cite{SS5} shows that the if $A$ has an admissible $G$-action then the CM-type of the corresponding skew group algebra $B=AG$ is  $\mathbb{D}_{\frac{N+1}{2}}$.

\begin{definition}
 We say that a 2-Calabi-Yau tilted algebra is  \emph{CM-minimal} if, for every vertex $i$, the reduction $A/J$ with $J=Ae_iA$ does not have the same CM-type as $A$.
\end{definition}
\smallskip

We now show the following result.

\begin{thm}
 \label{thm D}
 Let $A$ be a  dimer tree algebra with a nontrivial admissible $G$-action. Assume $A$ is  different from the algebra in (\ref{not Q}). Then the following are equivalent. 
 
 (a) $A$ is CM-minimal.
 
 (b)  The skew group algebra $B=AG$ is CM-minimal.
\end{thm}

\begin{proof} (a)$\Rightarrow$(b).
Assume $A$ is CM-minimal. To show (b), suppose to the contrary that there exists a vertex $i$ in $Q_B$ such that $B$ and $B/J$ have the same CM-type, where $J=Be_i B$.  We will distinguish three cases. 
 
 Suppose first that $i\ne 1^\pm,2^\pm$. Then $B/J$ is isomorphic to the skew group algebra of the reduction $A/A(e_i+e_{\zs i})A$ of $A$ relative to the $G$-orbit of $i$. Since $A$ is minimal, the CM-type of this reduction is strictly smaller than the CM-type of $A$. Consequently, the CM-type of $B/J$ is strictly smaller than the CM-type of $B$, a contradiction.

Before considering the next case, we make the following general observations.
There is a nonzero path from $2$ to $1$ in $A$, since $\za\colon 1 \to 2$ is an interior arrow and therefore it does not induce a zero relation in the Jacobian algebra $A$. Furthermore, every path $u$ from $2^-$ to $1^-$ in $B$ is zero, because $\za^-\colon 1^-\to 2^-$ is a boundary arrow in $Q_B$ and therefore induces a zero relation on the path $u$.
 Therefore there exists a nonzero path $w$ from $2^-$ to $1^+$ in $Q_B$. This means that the projective indecomposable $P_B(2^-)$ is supported at the vertex $1^+$. 
 
Now  suppose that $i=1^+$. 
 Since $B/J$ and $B$ have the same CM-type, and $P_B(2^-)$ is supported at $1^+$, then Corollary~\ref{cor} (d) implies that $P_B(1^+)\subset P_B(2^-)$. Thus the existence of the  arrow $\za^+\colon 1^+\to 2^+$ implies that $2^+$ lies in the support of $P_B(2^-)$. Consequently, there is a nonzero path from $2^-$ to $2^+$ in $B$. 
 But then there must be a nonzero path from 2 to 2 in  $A$, which is impossible since $A$ is Schurian, by  \cite[Corollary 3.33]{SS3}. 
 This completes the proof in the case where $i=1^+$. The case $i=1^-$ is symmetric.
 
 Suppose now $i=2^-$. 
The indecomposable projective module $P_B(1^-)$ is supported at $2^-$, because of the arrow $\za^-\colon 1^-\to 2^-$. Thus Corollary~\ref{cor}(d) implies that $P_B(2^-)\subset P_B(1^-)$. On the other hand, we have seen above that  there is a nonzero path $w$ from $2^-$ to $1^+$ in $B$. 
Hence $1^+$ is in the support of 
$P_B(2^-)$ and, since   
$P_B(2^{-})\subset P_B(1^-)$, this implies the existence of a nonzero path from $1^-$ to $1^+$ in $B$. Again we get a contradiction to the fact that $A$ is Schurian. 
This proves the case $i=2^-$ and the case $i=2^+$ is symmetric.

(b)$\Rightarrow$(a). Now assume $B$ is CM-minimal, and suppose that $A$ isn't. Let $i$ be a vertex in $A$ and $J=Ae_iA$ the corresponding 2-sided ideal such that $A/J$ and $A$ have the same CM-type. Thanks to Theorem~\ref{thm A reduction}, the vertex $i$ lies in a unique 3-cycle $C$ of the form $\xymatrix{h\ar[r]^\zb&i\ar[r]^\zg&j\ar@/^10pt/[ll]^\zd}$ with $\zb$ and $\zg$ boundary arrows and $\zd$ an interior arrow. Moreover, we have the weight conditions $\wt(\zg)=\overline{\wt}(\zb)=1$. 
Since $\zs$ fixes only the interior arrow $\za\colon 1\to 2$ of $Q$, it follows that none of $i,\zb,j$ are fixed under the action of $\zs$. 
 Thus there is a second 3-cycle  $\zs C$ that is given by $\xymatrix{\zs h\ar[r]^{\zs \zb}&\zs i\ar[r]^{\zs \zg}&\zs j\ar@/^10pt/[ll]^{\zs \zd}}$, and $\zs \zb,\zs \zg$ are boundary arrows and $\zs\zd $ is not.  Moreover $\wt(\zs\zg)=\overline{\wt}(\zs\zb)=1$. Letting $\zs J=Ae_{\zs i}A$, Theorem~\ref{thm A reduction} implies $A/\zs J$ and $A$ have the same CM-type.

We will now show that the double quotient $(A/J)/\zs J$ also has the same CM-type. Indeed, since $A$ is not 
the algebra given by the quiver  (\ref{not Q}), the two 3-cycles $C$ and $\zs C$ are disjoint. 
Therefore the vertex $\zs i$ still lies in the 3-cycle $\zs C$ in the quotient $A/J$. 

Furthermore the condition on the weight and the coweight of the arrows at $i$ and at $\zs i$ can be stated equivalently,  in terms of the  white regions of  the checkerboard polygon $\cals$ of \cite[Section 3.3.1]{SS3},  by saying that the two white regions $W(i,\zb)$ and $W(i,\zg)$ that contain the endpoints of the radical line $\rho(i)$ of vertex $i$ do not contain a boundary edge of $\cals$. Symmetrically, the two white regions $W(\zs i,\zs \zb),W(\zs i,\zs \zg)$ that contain the endpoints of the radical line $\rho(\zs i)$ of vertex $\zs i$ also don't contain a boundary edge of $\cals$. 

The group element $\zs\in G$ acts on $\cals$ as a rotation by angle $\pi$, see \cite[Proposition 4.11]{SS5}, and $\zs W(i,\zb)= W(\zs i,\zs \zb)$ and $\zs W(i,\zg)= W(\zs i,\zs \zg)$. In particular, the regions $W(i,\zb), W(\zs i,\zs\zb)$ are disjoint, and so are the regions $W(i,\zg), W(\zs i,\zs\zg)$.
Therefore the regions $W(\zs i,\zs \zb), W(\zs i,\zs\zg)$ still do not contain a boundary segment after the removal of the radical line $\rho (i)$ from $\cals$. In other words, $w(\zs \zg)=\overline{w}(\zs\zb)=1$ also holds in the quotient $A/J$. Again thanks to Theorem~\ref{thm A reduction}, the double quotient $(A/J)/\zs J$ has the same CM-type as $A/J$, and hence the same CM-type as $A$.

The algebra $(A/J)/\zs J$ is obtained by reducing the algebra $A$ at the two vertices $i$ and $\zs i$ that form a $G$-orbit. The reduced algebra $(A/J)/\zs J$ is still a dimer tree algebra with admissible $G$-action defined as the restriction of the $G$-action on $A$. The fact that  $A$ and $(A/J)/\zs J$ have the same CM-type implies that the corresponding skew group algebras $B=AG$ and $B'=((A/J)/\zs J)G$ also have the same CM-type. Furthermore, we have $B'=B/Be_iB$, where $i$ is the vertex in the quiver of $B$ that 
corresponds to the $G$-orbit $i,\zs i$ in the quiver of $A$.
Thus $B'$ is a reduction of $B$ that has the same CM-type as $B$, and this is a contradiction to part (b).
\end{proof}

\end{document}